\documentclass[a4paper, 12pt, abstract=true]{scrartcl}
\pdfoutput=1
\usepackage{amsmath,amssymb,amsthm}
\usepackage{mathtools}
\usepackage{bbm}
\usepackage[utf8]{inputenc}
\usepackage{fullpage}
\usepackage{enumerate}
\usepackage[shortlabels]{enumitem}
\usepackage{parskip}
\usepackage{verbatim}
\usepackage{hyperref}
\usepackage[
style=alphabetic,
sorting=nyt,
maxbibnames=99
]{biblatex}
\theoremstyle{definition}
\newtheorem*{definition*}{Definition}
\newtheorem{definition}{Definition}[section]
\newtheorem{lemma}[definition]{Lemma}
\newtheorem{theorem}[definition]{Theorem}
\newtheorem{fact}[definition]{Fact}
\newtheorem{remark}[definition]{Remark}
\newtheorem{corollary}[definition]{Corollary}

\DeclareMathOperator{\cf}{cf}
\let \succ \relax
\DeclareMathOperator{\succ}{succ}
\DeclareMathOperator{\supp}{supp}
\newcommand{\nunlhd}{%
	\mathrel{\ooalign{$\lneq$\cr\raise.22ex\hbox{$\lhd$}\cr}}}
\DeclareMathOperator{\dom}{dom}
\DeclareMathOperator{\splitt}{split}
\DeclareMathOperator{\htt}{ht^s}
\DeclareMathOperator{\cof}{cof}
\DeclareMathOperator{\ran}{ran}

\title{Strong Measure Zero Sets on $2^\kappa$ for $\kappa$ Inaccessible}
\author{Nick Steven Chapman and Johannes Philipp Sch\"urz \footnote{This author was generously supported by FWF project I3081.}}
\date{}

\begin{document}
\maketitle
\vspace{20pt}
\begin{abstract}
	We investigate the notion of strong measure zero sets in the context of the higher Cantor space $2^\kappa$ for $\kappa$ at least inaccessible. Using an iteration of perfect tree forcings, we give two proofs of the relative consistency of
	\[
	|2^\kappa| = \kappa^{++} + \forall X \subseteq 2^\kappa:\ X \text{ is strong measure zero if and only if } |X| \leq \kappa^+.
	\]
	Furthermore, we also investigate the stronger notion of stationary strong measure zero and show that the equivalence of the two notions is undecidable in ZFC.
\end{abstract}

\section*{Introduction}

In searching for a useful notion related to being a Lebesgue measure zero set, Borel \cite{borel_def} introduced strong measure zero sets.

\begin{definition*}
	A subset $X$ of the real line is strong measure zero iff for any sequence $(\varepsilon_n)_{n \in \omega}$ of positive real numbers there exists a sequence of intervals $(I_n)_{n \in \omega}$ with $\lambda(I_n) \leq \varepsilon_n$ and $X \subseteq \bigcup_{n \in \omega} I_n$.
\end{definition*}

Clearly, strong measure zero sets are measure zero and every countable set is strong measure zero. Moreover, it is also easy to see that perfect sets cannot be strong measure zero. It was conjectured by Borel that countability is perhaps the only constraint on strong measure zero sets, giving rise to the Borel Conjecture (BC): ``A set $X$ is strong measure zero if and only if $X$ is countable.''

In 1928, Sierpi\'nski \cite{sier} showed that CH implies the existence of uncountable strong measure zero sets (specifically, he showed that any Luzin set is strong measure zero). It was not until after the advent of Cohen's revolutionary technique of forcing that Laver \cite{laver_classic_proof} established the relative consistency (and thus independence from ZFC) of BC.

Over the years, investigations into matters related to strong measure zero sets (such as the interplay between BC and the size of the continuum \cite{bc_large_cont}, the dual notion of strongly meager sets \cite{dual_bc} and others) became testament to the fact that Borel's notion was indeed worthy of interest.

For our purposes the most interesting of these is Corazza's proof of the consistency of ``a set is strong measure zero iff it has size less than continuum'' (\cite{corazza}) in which he employs an $\omega_2$-length iteration of strongly proper forcings (a notion stronger than ``proper + $\omega^\omega$-bounding'' that includes well-known forcings such as Sacks and Silver), together with a previous result of Miller \cite{miller} to construct a model with 
\begin{quote}
	\centering ``Every set of reals of size continuum can be mapped uniformly continuously onto $[0,1]$''. 
\end{quote} 

We are interested in a version of Borel's Conjecture on higher cardinals $\kappa$. The higher Cantor space $2^\kappa$ and the higher Baire space $\kappa^\kappa$ come equipped with the standard ${<}\kappa$-box topology; see \cite{khomskii_kappa_reals} for basic properties of these spaces. Their elements are called $\kappa$-reals, or simply reals. Note that near universally, the assumption $\kappa^{<\kappa} = \kappa$ is made in discussions on the higher Baire space, without which the space exhibits some undesirable topological properties (see \cite[\S 2.1.]{higher_descriptive}). Especially in recent years, renewed interest has sparked among set theorists in studying these spaces; a compendium of open questions can be found in \cite{khomskii_questions}.

The following definition is due to Halko \cite{halko}:

\begin{definition*} 
	Let $X \subseteq 2^\kappa$. We call $X$ strong measure zero iff
	\[
	\forall f\in \kappa^\kappa \, \exists (\eta_i)_{i < \kappa}:\ \big (\, \forall i < \kappa:\ \eta_i \in 2^{f(i)} \, \big ) \land  X \subseteq \bigcup_{i< \kappa} [\eta_i].
	\]
\end{definition*}
This is a straightforward combinatorial reformulation (here $[\eta]$ is a basic clopen set as defined in the next section) of Borel's definition that is agnostic towards the existence of a measure on $2^\kappa$. Let $\mathcal{SN}$ be the collection of all strong measure zero sets; it is easy to see that $\mathcal{SN}$ is a proper, ${\leq}\kappa$-complete ideal (see also Lemma \ref{lem: cover unboundedly}) on $2^\kappa$ containing all singletons. 

The Borel Conjecture on $\kappa$ (BC($\kappa$)) is the statement ``a subset of $2^\kappa$ is strong measure zero iff it has cardinality ${\leq}\kappa$''. Strong measure zero sets for $\kappa$ regular uncountable have been studied in \cite{halko_shelah}, where the authors have proven that BC($\kappa$) is false for successor $\kappa$ satisfying $\kappa^{<\kappa} = \kappa$. 

Throughout this paper we shall restrict our attention to $\kappa$ at least inaccessible, thus in particular $\kappa^{<\kappa} = \kappa$. The question of the consistency of BC($\kappa$) on such $\kappa$ is still open \cite{khomskii_questions}. An argument similar to the one in the proof of Theorem \ref{th: easy inclusion aux} - forgoing the notion of $\kappa^\kappa$-bounding and focusing instead directly on encoding coverings within the $\kappa$-Cohen real - tells us that a $\kappa^{++}$-c.c. forcing iteration of length $\kappa^{++}$ in which $\kappa$-Cohen reals are added cofinally will necessarily yield large strong measure zero sets - in fact, the set of ground model $\kappa$-reals will become strong measure zero. Unfortunately, by the results in \cite{laver_trees_baire_space}, every Laver-like tree forcing on $\kappa^\kappa$ necessarily adds a $\kappa$-Cohen real. Any treatment of the consistency of BC$(\kappa)$ thus cannot be merely a straightforward adaptation of Laver's results; potentially, a wholly new approach is needed here.

We shall give two proofs establishing the relative consistency of
\[
\text{ZFC} + |2^\kappa| = \kappa^{++} + \mathcal{SN} = [2^\kappa]^{\leq \kappa^+},
\]
the first of which is an adaptation of an iteration found in \cite{smz_no_cohen} and requires $\kappa$ to be strongly unfoldable (a large cardinal property between weakly compact and Ramsey that is consistent with $V=L$). The second, somewhat better, proof only requires $\kappa$ to be inaccessible and employs the same iteration by establishing minimality of the respective forcing extension, following the approach of Corazza \cite{corazza}.

Last but not least, we would like to thank Martin Goldstern for fruitful discussions during the preparation of this paper. We would also like to thank the anonymous referee for their many valuable remarks and suggestions that helped substantially improve the presentation of the paper.

\section*{Notation and Basic Definitions}

Let us make some preliminary remarks.

The higher Cantor space $2^\kappa$ is equipped with the standard ${<}\kappa$-box topology, whose base consists of the basic clopen sets 
\[
[\eta] := \{b \in 2^\kappa: \eta \lhd b\}
\]
for $\eta \in 2^{<\kappa}$; for the higher Baire space $\kappa^\kappa$ the topology is defined analogously. The relation $\eta \lhd \nu$ denotes the extension relation for sequences, i.e. $\eta = \nu {\restriction} i$ for some $i \leq \dom(\nu)$. The relation $\eta \bot \nu$ denotes incompatibility, i.e. $\eta \ntriangleleft \nu$ and $\nu \ntriangleleft \eta$.

A ($\kappa$-) tree is a subset of $\kappa^{<\kappa}$ closed under initial segments.

Let $T \subseteq \kappa^{<\kappa}$ be a tree and $\eta \in T$. Then we define the following notions:
\begin{itemize}
	\item A $b \in \kappa^\kappa$ is a \textit{branch} of $T$ iff $b {\restriction} i \in T$ for all $i < \kappa$. Let $[T]$ denote the set of all branches of $T$.
	\item Denote by $\succ_T(\eta)$ the set of immediate successors of $\eta$ in $T$. Call $\eta$ a \textit{splitting node} of $T$ iff $|\succ_T(\eta)| > 1$. Denote the set of all splitting nodes of $T$ as $\splitt(T)$. We will only consider trees in which every node has a successor.
	\item $T$ is \textit{perfect} iff for every $\eta \in T$ there is a $\nu$ such that $\eta \lhd \nu$ and $\nu \in \splitt(T)$. Note that for $\kappa \neq \omega$ this is not equivalent to $[T]$ containing a homeomorphic copy of $2^\kappa$, even if every node of the tree lies on a branch (of length $\kappa$).
	\item The \textit{splitting height} $\htt_T(\eta)$ of a node $\eta$ is the order type of the set $\{\nu \nunlhd \eta:\ \nu \in \splitt(T)\}$. Additionally, for $i < \kappa$, define 
	\[
	\splitt_i(T) := \{ \eta \in \splitt(T): \htt_T(\eta) =i \}.
	\] 
	\item As usual, the set of branches of a tree is a closed set and every closed set $Y$ can be represented as the set of branches of the tree $T = \{b {\restriction} i:\ i < \kappa \land b \in Y\}$. However, it may be the case that this tree $T$ necessarily contains \textit{dying branches}, i.e. $T$ might contain an increasing sequence $(\eta_i)_{i < \lambda}$ with $\lambda < \kappa$ whose limit $\bigcup_{i < \lambda} \eta_i$ is not an element of $T$ \footnote{Consider for example the closed set $2^\kappa \backslash [\eta]$, where $\eta \in 2^\omega$.}. This phenomenon is unique to the $\kappa$-case and has no $\omega$-equivalent. 
	
	We say $T$ (or $[T]$) is \textit{superclosed} iff this does not happen, meaning that whenever $\lambda < \kappa$ is a limit ordinal and $\eta \in \kappa^\lambda$, then $\eta \in T \Leftrightarrow \forall i < \lambda:\ \eta {\restriction} i \in T$. \label{def: superclosed}
\end{itemize}

We shall attempt to, wherever feasible, adhere to certain self-imposed notational conventions. In this vein, the letters $i,j,k,\ell,m$ will generally refer to ordinals ${<}\kappa$; $\delta, \lambda$ to limit ordinals ${\leq}\kappa$ and $\alpha, \beta, \gamma, \zeta$ to ordinals ${\leq} \kappa^{++}$. The letters $p,q,r,s,t$ denote conditions while $\eta, \nu, \rho$ are elements of $\kappa^{<\kappa}$. The pair $F,i$ will always fulfil $F \in [\alpha]^{<\kappa}, i < \kappa$, where $\alpha \leq \kappa^{++}$ is either explicitly given or clear from context.
\clearpage

\section{Perfect Tree Forcing}

We are interested in a particular forcing consisting of ${<}\kappa$-splitting perfect trees whose splitting is bounded by an $f \in \kappa^\kappa$ with $f(i) \geq 2$ for all $i < \kappa$.
\begin{definition}
	Let $p \in PT_f$ iff 
	\begin{enumerate}[label=(S\arabic*)]
		\item $p \subseteq \kappa^{< \kappa}$ is a nonempty tree
		\item $p$ is perfect \label{ax: splitting unbounded} 
		\item $\forall \eta \in p \, \forall i \in \text{dom} (\eta):\ \eta(i) < f(i)$
		\item $p$ has full splitting: $\forall \eta \in p \colon \vert \succ_p (\eta) \vert =1 \vee 
		\succ_p (\eta)= \{\eta ^\frown j \colon j < f(\text{dom}\, \eta)\}$ \label{ax: full splitting}
		\item $p$ is superclosed \label{ax: superclosed}
		\item splitting is continuous: If $\lambda < \kappa$ is a limit, then \\ $\forall \eta \in \kappa^{\lambda} \cap p:\ \{\nu \nunlhd \eta \colon \nu \in \splitt(p)\} \text{ is unbounded in } \eta \Rightarrow \eta \in \splitt(p)$ \label{ax: splitting continuous}
	\end{enumerate}
\end{definition} 

The significance of \ref{ax: full splitting} and \ref{ax: splitting continuous} lies in ensuring ${<}\kappa$-closure of the forcing (see Lemma \ref{lem: <k closed}). The axioms \ref{ax: full splitting} and \ref{ax: superclosed} guarantee that for all $\eta \in p$ we have 
\[
[\eta] \cap [p] \neq \emptyset,
\]
i.e. there is a branch of $p$ going through $\eta$. Under the other axioms, \ref{ax: splitting unbounded} + \ref{ax: splitting continuous} is equivalent to the following statement: whenever $b \in [p]$ is a branch of $p$, then
\[
\{i < \kappa: b{\restriction} i \in \splitt(p)\}
\]
is a club subset of $\kappa$. 

For $f \equiv 2$ we have a $\kappa$-version of Sacks forcing, first studied by Kanamori \cite{kanamori}. An overview of variants of familiar forcing notions on higher cardinals can be found in \cite{khomskii_kappa_reals}.

The rest of this section is devoted to proving some regularity properties for $PT_f$, generalized straightforwardly from the classical treatment of similar tree forcings on $\omega^\omega$.

Set $q \leq_{PT_f} p$ iff $q \subseteq p$. For a $PT_f$-generic filter $G$ define the generic real $s_G$ to be the unique real contained in $\bigcap_{p \in G} [p]$.

\begin{fact}
	For a condition $p \in PT_f$ the set $\splitt_i(p)$ is a front in $p$, i.e. it is an antichain in $(p, \lhd)$ with
	\[
	\forall b \in [p]: |b \cap \splitt_i(p)| = 1.
	\]
	Call it the $i$-th splitting front of $p$. 
\end{fact}

\begin{lemma}
	Let $i < \kappa$ and $p \in PT_f$ be a condition. Then $|\splitt_i(p)| < \kappa$.
\end{lemma}

\begin{proof}
	We proceed by induction on $i$: 
	\begin{itemize}
		\item $i = 0$: Trivial.
		\item $i \to i+1$: The map $\eta \mapsto \min\{\nu \lhd \eta: \htt_p(\nu) = i+1\}$ is bijection between $\splitt_{i+1}(p)$ and $\bigcup_{\eta \in \splitt_i(p)} \succ_p(\eta)$. By the inductive hypothesis and the fact that $p$ is ${<}\kappa$-splitting, the latter set has size $<\kappa$.
		\item $\lambda$ is a limit: Since every $\eta \in \splitt_\lambda(p)$ is the limit of a sequence $(\eta_j)_{j < \lambda}$ with $\eta_j \in \splitt_j(p)$, we have $|\splitt_\lambda(p)| \leq |\prod_{j < \lambda} \splitt_j(p)| < \kappa$ by the inaccessibility of $\kappa$.
	\end{itemize}
\end{proof}

\begin{definition}
	Let $(\mathcal{P}, \leq_\mathcal{P})$ be a forcing notion and $(\leq_i)_{i <\kappa}$ be a sequence of reflexive and transitive binary relations on $\mathcal{P}$ such that
	\[
	\forall j < i < \kappa \colon (\leq_i) \, \subseteq \, (\leq_j) \, \subseteq (\leq_\mathcal{P}).
	\]
	Then
	\begin{enumerate}
		\item $(p_j)_{j<\delta}$ is a \textit{fusion sequence} of length $\delta \leq \kappa$ iff $\forall j < k < \delta:  p_k \leq_j p_j$.
		\item $\mathcal{P}$ has \textit{Property B} iff
		\begin{itemize}
			\item $(\mathcal{P}, \leq_\mathcal{P})$ is ${<}\kappa$-closed.
			\item Whenever $(p_j)_{j<\delta}, \delta \leq \kappa$ is a fusion sequence in $\mathcal{P}$, then there exists a \textit{fusion limit} $q$ with $\forall j < \delta: q \leq_j p_j$.
			\item If $A$ is a maximal antichain, $p \in \mathcal{P}$ and $i < \kappa$, then there exists a $q \leq_ i p$ such that $A {\restriction} q:=\{ r \in A \colon r \parallel q\}$ has size ${<}\kappa$, where $\parallel$ means compatible. 
		\end{itemize}
	\end{enumerate}
\end{definition}

Equivalently, we can demand the third condition to hold for all antichains $A$, by enlarging them to maximal antichains if necessary. Note that by weakening the third requirement to $|A {\restriction} q| \leq \kappa$, we get a $\kappa$-version of Baumgartner's Axiom A. Property B is thus a variant of Axiom A combined with the notion of being $\kappa^\kappa$-bounding \cite[Def. 7.2.C]{barto_judah}; it is well-known from the countable context that many standard tree forcings, such as Sacks and Silver forcing, have this property.

\begin{lemma}
	Property B implies $\kappa^\kappa$-bounding.
\end{lemma}

\begin{proof}
	Assume $p \Vdash \dot{g} \in \kappa^\kappa$ and $\dot{g}(i)$ is decided by an antichain $A_{i+1}$ for every $i < \kappa$. Construct a fusion sequence $(q_{i})_{i < \kappa}$ below $p$ by setting $q_0 := p$ and finding a $q_{i+1} \leq_i q_i$ with $|A_{i+1} {\restriction} q_{i+1}| < \kappa$ in successor steps. In limit steps $\lambda$, set $q_\lambda$ to be a fusion limit of $(q_i)_{i < \lambda}$. The fusion limit $q_\kappa$ of the whole sequence will force $q_\kappa \Vdash \dot{g} \leq \check{h}$ for some $h \in \kappa^\kappa$ in the ground model.
\end{proof}

\begin{lemma} \label{lem: <k closed}
	$PT_f$ is ${<}\kappa$-closed.
\end{lemma}

\begin{proof}
	If $(p_i)_{i < \delta}$ with $\delta < \kappa$ is a decreasing sequence, set $q := \bigcap_{i < \delta} p_i$. We check that $q$ is a condition; only \ref{ax: splitting unbounded} is nontrivial, so we assume that all other axioms hold.
	
	Let thus $\eta \in q$. For some $b \in [q]$ with $\eta \lhd b$ (recall that by \ref{ax: full splitting} + \ref{ax: superclosed} such a $b$ exists) consider the sets 
	\[
	C_i := \{j < \kappa : b {\restriction} j \in \splitt(p_i)\}.
	\]
	By \ref{ax: splitting unbounded} and \ref{ax: splitting continuous}, $C_i$ is a club subset of $\kappa$. Thus $\bigcap_{i < \delta} C_i$ is a club and yields a $\nu$ with $\eta \lhd \nu$ and $\nu \in \splitt(q)$.
\end{proof}

\begin{remark}
	Clearly, the intersection $\bigcap_{i < \delta} p_i$ in the previous lemma is simultaneously also the greatest lower bound of the decreasing sequence $(p_i)_{i < \delta}$, $\delta < \kappa$.
\end{remark}

\begin{definition}
	For $p,q \in PT_f$, define $q \leq_i p$ iff $q \leq_{PT_f} p$ and $\splitt_i(p) = \splitt_i(q)$.
\end{definition}

\begin{fact}
	The following are equivalent:
	\begin{enumerate}
		\item $q \leq_i p$
		\item $q \leq_{PT_f} p$ and $\forall j \leq i: \splitt_j(p) = \splitt_j(q)$
		\item $q \leq_{PT_f} p$ and $\forall \eta \in p: \htt_p(\eta) \leq i \Rightarrow \succ_p(\eta) \subseteq q$
		\item $q \leq_{PT_f} p $ and $\splitt_{i+1}(p) \subseteq q$
	\end{enumerate}
\end{fact}

It remains to prove that equipped with these relations, $PT_f$ has Property B.

\begin{lemma}
	For every fusion sequence $(p_j)_{j<\delta}$ of length $\delta \leq \kappa$ in $PT_f$ there exists a $q$ with $\forall j < \delta: q \leq_j p_j$.
\end{lemma}

\begin{proof}
	If $\delta < \kappa$, the intersection $q$ from Lemma \ref{lem: <k closed} can be seen to also be a fusion limit. 
	
	Otherwise once again set $q = \bigcap_{j < \kappa} p_j$ and follow the proof of Lemma \ref{lem: <k closed}; along a branch $b \in [q]$ again define the sets 
	\[
	C_j := \{\ell < \kappa : b {\restriction} \ell \in \splitt(p_j)\}.
	\]
	By using the fact that $(p_j)_{j < \kappa}$ is a fusion sequence, one can arrive at
	\[
	\left(\Delta_{j < \kappa} C_j\right) \cap \{\lambda < \kappa: \lambda \text{ limit}\} \subseteq \bigcap_{j < \kappa} C_j,
	\]
	which is enough to conclude that $\bigcap_{j < \kappa} C_j$ is also a club by the closure of the club filter under diagonal intersections. It can easily be seen that $q$ is a fusion limit.
\end{proof}

Before concluding the proof, we first give two definitions which will come in handy later in the iteration context.

\begin{definition}
	For a condition $p \in PT_f$ and $\eta \in p$, define $p^{[\eta]} := \{\nu \in p: \nu \lhd \eta \vee \eta \lhd \nu\}$.
\end{definition}

One can see easily that $p^{[\eta]}$ is a stronger condition than $p$ and that for any $i < \kappa$ we have $p = \bigcup_{\eta \in \splitt_i(p)} p^{[\eta]}$. 

\begin{definition} \label{def:choice set}
	Let $p \in PT_f$ be a condition and $i < \kappa$. We say that a condition $s \in PT_f$ is \textit{$(p,i)$-determined} iff $s \leq p$ and
	\[
	|s \cap \splitt_i(p)| = 1.
	\]
\end{definition}

\begin{lemma} \label{lem:choice set dense}
	The set of $(p,i)$-determined conditions is dense below $p$ for all $i$.
\end{lemma}

\begin{proof}
	For any $r \leq p$ we may extend the stem of $r$ in the following way: take any branch $b \in [r] \subseteq [p]$; since we then know $|b \cap \splitt_{i}(p)| = 1$, we see that there is a unique $\nu$ with $\nu \in b \cap r \cap \splitt_i(p)$. Then $r^{[\nu]}$ is $(p,i)$-determined. 
\end{proof}

\begin{theorem}
	$PT_f$ has Property B.
\end{theorem}

\begin{proof}
	It remains to show the antichain condition. To this end, let $A$ be a maximal antichain, $p \in PT_f$ and $i < \kappa$. Enumerate $\splitt_{i+1}(p)$ as $(\eta_j)_{j < \delta}$ with $\delta < \kappa$. We will decompose $p$ into $|\delta|$ many parts, each of which will be thinned out above the $(i+1)$-th splitting front.
	
	Proceed by finding for each $j < \delta$ a condition $s_j \leq p^{[\eta_j]}$ such that $|A {\restriction} s_j| = 1$. Set 
	\[	
	q := \bigcup_{j < \delta} s_j.
	\]
	Then $q \in PT_f$ is a condition with $\splitt_{i+1}(p) \subseteq q$ and thus $q \leq_i p$. To prove $|A {\restriction} q| < \kappa$, let $r \in A$ be compatible with $q$. By the previous lemma we may pick a $t_r$ that is $(p,i+1)$-determined with $t_r \leq r,q$ and hence $t_r \cap \splitt_{i+1}(p) = \{\eta_{j_r}\}$ for some $j_r < \delta$. But since $t_r \leq q$, we can conclude $t_r \leq s_{j_r}$ and thus $r \parallel s_{j_r}$. We have thus found a function from $A {\restriction} q$ to $\delta$, mapping $r \mapsto j_r$, which is injective (since $|A {\restriction} s_j| = 1$ for all $j < \delta$). The desired conclusion $|A {\restriction} q| < \kappa$ follows.
\end{proof}

\section{The Iteration}

The backbone of our forcing construction will consist of an iteration of $PT_f$ forcings. Let therefore $\langle \mathbb{P}_\alpha, \dot{\mathbb{Q}}_\beta : \alpha \leq \kappa^{++}, \beta < \kappa^{++} \rangle$ be a ${\leq}\kappa$-supported forcing iteration with 
\[
\Vdash_{\mathbb{P}_\alpha} \dot{\mathbb{Q}}_\alpha = PT_{f_\alpha}
\]
where the sequence $(f_\alpha)_{\alpha < \kappa^{++}}$ is in the ground model and $f_\alpha(i) \geq 2$ for all $i < \kappa$. Set $\mathbb{P} := \mathbb{P}_{\kappa^{++}}$.

As a matter of notation, let $\dot{G}_\alpha$ for $\alpha \leq \kappa^{++}$ denote the canonical $\mathbb{P}_\alpha$-name for a $\mathbb{P}_\alpha$-generic filter; we also write $\dot{G}$ for $\dot{G}_{\kappa^{++}}$. Finally, let $\dot{s}_\alpha$ be the canonical name for the $\alpha$-th generic real.

This section is dedicated to verifying some regularity properties of such iterations. We will observe that
\begin{enumerate}
	\item $\mathbb{P}$ is ${<}\kappa$-closed
	\item $\mathbb{P}$ does not collapse $\kappa^+$
	\item if $V \models |2^\kappa| = \kappa^+$, then $\mathbb{P}$ has the $\kappa^{++}$-c.c.,
\end{enumerate}
thus in aggregate no cardinals are collapsed when forcing with $\mathbb{P}$.

\begin{fact}
	$\mathbb{P}$ is ${<}\kappa$-closed.
\end{fact}

In the countable case, the favoured tool one would look towards in the endeavour of preserving $\omega_1$ is the notion of properness. Finding a satisfactory analogue for higher cardinals is a long-standing open problem (see e.g. \cite{kappa_proper1} and \cite{kappa_proper2}). A relatively straightforward generalization that still enjoys many desirable qualities of properness is the following:

\begin{definition}
	A forcing $\mathcal{P}$ is called $\kappa$-proper iff for every sufficiently large $\theta$ (e.g. $\theta > \vert 2^{\mathcal{P}}\vert$) and every elementary submodel $M \preccurlyeq H(\theta)$ such that $\mathcal{P} \in M$, $|M| = \kappa$ and $^{<\kappa}M \subseteq M$, and every $p \in \mathcal{P} \cap M$, there exists $q \leq_\mathcal{P} p$ such that for every dense $D \in M$, $D \cap M$ is predense below $q$.
\end{definition}

\begin{fact}
	Forcing notions that are ${<}\kappa^+$-closed or have the $\kappa^+$-c.c. are $\kappa$-proper. Furthermore, $\kappa$-proper forcing notions do not collapse $\kappa^+$.
\end{fact}

Further details on $\kappa$-properness can be found in \cite{khomskii_kappa_reals}.

Unfortunately, in stark contrast to the classical setting, there is no preservation theorem for $\kappa$-properness in iterations (see \cite{roslanowski} for an iteration of $\kappa^+$-c.c. forcings whose $\omega$-limit collapses $\kappa^+$). Our strategy for ensuring $\kappa$-properness is to verify an iteration version of Property B. Similar to fusion with countable support, in such cases the correct tool is the following notion: 

\begin{definition}
	For $\zeta \leq \kappa^{++}$ let $\langle \mathcal{P}_\alpha, \dot{\mathcal{Q}}_\beta : \alpha \leq \zeta, \beta < \zeta \rangle$ be a ${\leq}\kappa$-support iteration with 
	\[
	\forall \alpha< \zeta \colon  \Vdash_{\alpha} `` \ \dot{Q}_\alpha \text{ has Property B ''}.
	\]
	Let $F \in [\zeta]^{<\kappa}$ and $i < \kappa$. We define $ q \leq_{F,i} p$ iff 
	\[
	q \leq_{\mathcal{P}_{\zeta}} p \text{ and } \forall \beta \in F \colon q {\restriction} \beta \Vdash_\beta q(\beta) \,  \leq_i^{\dot{\mathcal{Q}}_\beta} \, p(\beta).
	\]
	Then
	\begin{enumerate}
		\item A sequence $\langle p_i, F_i: i < \delta \rangle$ of length $\delta \leq \kappa$ is called a fusion sequence iff		
		\begin{itemize}
			\item $\forall j < k < \delta: p_k \leq_{F_j,j} p_j$
			\item The $F_j$ are increasing and, if $\delta = \kappa$, then $\bigcup_{j < \delta} \text{supp}(p_j) \subseteq \bigcup_{j < \delta} F_j$.
		\end{itemize}
		\item We say that $\mathcal{P}_{\zeta}$ has \textit{Property B*} iff
		\begin{itemize}
			\item For every fusion sequence $\langle p_i, F_i: i < \delta \rangle,\ \delta \leq \kappa$ there exists a fusion limit $q$ with $\forall j < \delta: q \leq_{F_j,j} p_j$.
			\item For every maximal antichain $A$, every $p \in \mathcal{P}_{\zeta}, F \in [\zeta]^{<\kappa}$ and $i<\kappa$ there exists a $q \leq_{F,i} p$ such that $|A {\restriction} q| < \kappa$. \label{def:antichain condition}
		\end{itemize}
	\end{enumerate}
\end{definition}
Hence for iterations we consider fusion sequences pointwise, with the added caveat of being able to delay fusion arbitrarily long in each coordinate. In practice, the auxiliary sets $F_j$ will almost always be defined by a bookkeeping argument relative to the $p_j$.

\begin{fact} \label{fact: b* implies proper and bounding}
	Property B* implies $\kappa$-properness and $\kappa^\kappa$-bounding.
\end{fact}

In the definition of Property B*, only the antichain condition is nontrivial. In fact, for such iterations of Property B forcings, fusion limits always exist.
\begin{lemma}
	With the notation from the previous definition, every fusion sequence $\langle p_i, F_i: i < \delta \rangle,\ \delta \leq \kappa$ in $\mathcal{P}_{\zeta}$ has a fusion limit $q$. \label{lem: iterated fusion limits exist}
\end{lemma}

\begin{proof}
	We construct $q$ inductively such that $\mathcal{P}_\alpha \ni q {\restriction} \alpha$ is a fusion limit of $\langle p_i {\restriction} \alpha , F_i \cap \alpha : i < \delta \rangle$ for each $\alpha \leq \zeta$.
	
	Assume $q {\restriction} \alpha$ has been defined for $\alpha < \zeta$. To define $q(\alpha)$, distinguish three cases:
	\begin{itemize}
		\item $\alpha \in \bigcup_{j < \delta} \supp(p_j) \wedge \alpha \in \bigcup_{j < \delta} F_j$: Find $j^*(\alpha)$ minimal such that $\alpha \in F_{j^*(\alpha)}$. Now $q {\restriction} \alpha \Vdash `` (p_j(\alpha))_{j \geq j^*(\alpha)} \text{ is a fusion sequence''}$, so let $q(\alpha)$ be a fusion limit of that sequence.
		\item $\alpha \in \bigcup_{j < \delta} \supp(p_j) \wedge \alpha \notin \bigcup_{j < \delta} F_j$: Note that this case may only occur for $\delta < \kappa$, thus we may use ${<}\kappa$-closure of $\dot{\mathcal{Q}}_\alpha$ to construct $q(\alpha)$ from $(p_j(\alpha))_{j < \delta}$.
		\item $\alpha \notin \bigcup_{j < \delta} \supp(p_j)$: Set $q(\alpha) := \mathbbm{1}_{\dot{\mathcal{Q}}_\alpha}$.
	\end{itemize}
	
	To see that $q {\restriction} \gamma \in \mathcal{P}_\gamma$ for limit $\gamma$, merely note $\supp(q {\restriction} \gamma) \subseteq \bigcup_{i < \delta} \supp(p_i {\restriction} \gamma)$.
\end{proof}

\begin{remark} \label{rem: weakest fusion limit}
	Note that the forcings $\dot{\mathcal{Q}}_\alpha = PT_{f_\alpha}$ fulfil ${<}\kappa$-closure and the existence of fusion limits in a particularly strong way: in either case, a canonical weakest lower bound/fusion limit exists. Thus by following the above proof and choosing these canonical conditions, we can see that an iteration of $PT_f$ forcings also fulfils a stronger fusion condition: for every fusion sequence there exists a canonical, weakest fusion limit.
\end{remark}

Some work remains to prove the antichain condition for $\mathbb{P}_\zeta$, which we do in a rather ad hoc manner by induction on $\zeta$. On the way we will introduce some notation that will also come in handy later.

First off, let us define the iteration version of Definition \ref{def:choice set} and the corresponding density lemma.
\begin{definition}
	Let $\zeta \leq \kappa^{++}, p \in \mathbb{P}_\zeta, F \in [\zeta]^{<\kappa}$ and $i < \kappa$. We say a condition $s \in \mathbb{P}_\zeta$ is $(p,F,i)$\textit{-determined} following $g \in \prod_{\beta \in F} \kappa^{<\kappa}$ iff $s \leq_{\mathbb{P}_\zeta} p$ and
	\begin{gather*}
		\forall \beta \in F \, \exists \eta_\beta \in \kappa^{<\kappa}: \\
		s {\restriction} \beta \Vdash s(\beta) \cap \splitt_{i}(p(\beta)) = \check{\{\eta_\beta\}} \wedge \succ_{s(\beta)}(\eta_\beta) = \check{\{g(\beta)\}}.
	\end{gather*}
	
	We say a condition $s$ is $(p,F,i)$-determined iff it is $(p,F,i)$-determined following some (unique) $g$.
\end{definition}

The function $g$ prescribes the choices $s$ makes at the $i$-th splitting front of $p$; it is completely determined by $s$. 

\begin{lemma}
	The set of $(p,F,i)$-determined conditions is dense below $p \in \mathbb{P}_\zeta$ for all $p,F,i$ and the set of $(p,F,i)$-determined conditions following $g$ is open for all $p,F,i,g$. 
\end{lemma}

\begin{proof}
	Enumerate $F$ as an increasing sequence $(\beta_j)_{j < \delta}$ with $\delta < \kappa$ and set $\beta_\delta := \zeta$. For an $r \leq p$ we will inductively construct a decreasing sequence $(s_j)_{j \leq \delta}$ below $r$ and a $\subseteq$-increasing sequence $(g_j)_{j \leq \delta}$ with $g_j \in \prod_{\beta \in F \cap \beta_j} \kappa^{<\kappa}$ such that $s_j \text{ is } (p, F \cap \beta_j, i) \text{-determined}$ following $g_j$. 
	
	\begin{itemize}
		\item $j=0$: Set $s_{0} := r$.
		\item $j \to j+1$: Since $s_j {\restriction} \beta_j \Vdash s_j(\beta_j) \leq_{\dot{\mathbb{Q}}_{\beta_j}} p(\beta_j)$, we may use Lemma \ref{lem:choice set dense} to find $\mathbb{P}_{\beta_j}$-names $\dot{t}, \dot{\eta}_{\beta_j}, \dot{\nu}_{\beta_j} $ with 
		\[
		s_j {\restriction} \beta_j \Vdash \dot{t} \in \mathbb{Q}_{\beta_j} \wedge \dot{t} \leq_{\dot{\mathbb{Q}}_{\beta_j}} s_j(\beta_j)
		\] 
		and 
		\begin{gather*}
			s_j {\restriction} \beta_j \Vdash \dot{t} \cap \splitt_i(p) = \{\dot{\eta}_{\beta_j}\} \wedge \succ_{\dot{t}}(\dot{\eta}_{\beta_j}) = \{\dot{\nu}_{\beta_j}\}.
		\end{gather*}
		Find a stronger condition $\tilde{s}_j \leq s_j {\restriction} \beta_j$ that decides the names $\dot{\eta}_{\beta_j}, \dot{\nu}_{\beta_j}$ as $\eta_{\beta_j}, \nu_{\beta_j}$. Define $s_{j+1} := \tilde{s}_j ^\frown \dot{t} ^\frown (s_j {\restriction} [\beta_j + 1, \zeta))$ and $g_{j+1} := g_j \cup \{(\beta_j, \nu_{\beta_j})\}$.
		\item $\lambda \leq \delta$ is a limit: By ${<}\kappa$-closure we can find a lower bound $s_\lambda$ of the sequence $(s_\ell)_{\ell < \lambda}$. Define $g_\lambda := \bigcup_{\ell < \lambda} g_\ell$. Clearly, $s_\lambda$ is $(p, F \cap \beta_\lambda, i)$-determined following $g_\lambda$.
	\end{itemize}
	
	Now $s_\delta \leq r$ is $(p,F,i)$-determined following $g_\delta$.	Lastly, if $s$ is $(p,F,i)$-determined following $g$, then clearly any $s' \leq s$ is as well.
\end{proof}

\begin{fact} \label{fact: determined p p'}
	If $p' \leq_{F,i} p$ and $s \leq p'$, then $s$ is $(p,F,i)$-determined iff it is $(p', F,i)$-determined.
\end{fact}

Suppose now that $s \leq_{PT_f} p$. The extension of $p$ to $s$ may be undertaken in two steps by interpolating on the $\leq_i$ relation. In the first step, we thin out as much as is necessary from $p$, but only in its `upper regions' - say, above the $(i+1)$-th splitting front - yielding an interpolating condition $p^{(s)}$ with $p^{(s)} \leq_i p$ (above nodes not present in $s$, $p$ may be left untouched in the extension to $p^{(s)}$). In the second step, nodes are removed from $p^{(s)}$, but only near the base of the tree, such that whenever $\eta \in p^{(s)} \backslash s$, then there is already some initial segment $\nu \lhd \eta$ with $\nu \in p^{(s)} \backslash s$ and $\htt_{p^{(s)}}(\nu) \leq i+1$. We thus have
\[
s \leq p^{(s)} \leq_i p.
\]
This motivates the next lemma.

\begin{lemma}[Interpolation] \label{lem: interpolation}
	Let $p \in \mathbb{P}_\zeta$ and $s$ be $(p,F,i)$-determined following $g \in \prod_{\beta \in F} \kappa^{<\kappa}$ for some $F \in [\zeta]^{<\kappa}, i < \kappa$. Then there exists a condition $p^{(s)} \leq_{F,i} p$ with
	\begin{itemize}
		\item $s \leq_{\mathbb{P}_\zeta} p^{(s)} \leq_{F,i} p$ and
		\item for all $(p,F,i)$-determined conditions $s'$ following $g$, whenever $s' \leq_{\mathbb{P}_\zeta} p^{(s)}$, then already $s' \leq_{\mathbb{P_\zeta}} s$.
	\end{itemize}
\end{lemma}
\begin{proof}
	Construct $p^{(s)}$ by induction such that for each $\alpha \leq \zeta$ we have $p^{(s)} {\restriction} \alpha \in \mathbb{P}_\alpha$ and $p^{(s)} {\restriction} \alpha \leq_{F \cap \alpha, i} p {\restriction} \alpha$.
	
	Assume $p^{(s)} {\restriction} \alpha$ has been defined; to define $p^{(s)}(\alpha)$, there are two cases to distinguish:
	\begin{itemize}
		\item If $\alpha \notin F$, set $p^{(s)}(\alpha) := 
		\begin{cases}
			s(\alpha) & \text{ if } s {\restriction} \alpha \in \dot{G}_\alpha  \\
			p(\alpha) & \text{ otherwise.}
		\end{cases}$
		\item If $\alpha \in F$, set $p^{(s)}(\alpha) :=
		\begin{cases}
			s(\alpha) \cup (p(\alpha) \backslash p(\alpha)^{[g(\alpha)]}) & \text{ if } s {\restriction} \alpha \in \dot{G}_\alpha \\
			p(\alpha) & \text{ otherwise.}
		\end{cases}$
		
		Note that we have $s {\restriction} \alpha \Vdash g(\alpha) \in p(\alpha)$ and
		\[
		p^{(s)} {\restriction} \alpha \Vdash p^{(s)}(\alpha) \leq_{i} p(\alpha).
		\]
	\end{itemize}
	To see that $p^{(s)} {\restriction} \gamma \in \mathbb{P}_\gamma$ for $\gamma$ limit, we note that 
    \[
        \supp(p^{(s)} {\restriction} \gamma) \subseteq \supp(s {\restriction} \gamma).
    \]
     Furthermore, we clearly have $s \leq p^{(s)}$. 
	
	It remains to check the second requirement. Take some $(p,F,i)$-determined $s'$ following $g$ with $s' \leq p^{(s)}$. Assume inductively that $s' {\restriction} \alpha \leq s {\restriction} \alpha$. Since the case $\alpha \notin F$ is trivial, we may restrict our attention to the case $\alpha \in F$. Then we have $s' {\restriction} \alpha \Vdash s'(\alpha) \leq_{\mathbb{Q}_\alpha} p^{(s)}(\alpha) = s(\alpha) \cup (p(\alpha) \backslash p(\alpha)^{[g(\alpha)]})$. But then we already have $s' {\restriction} \alpha \Vdash s'(\alpha) \leq_{\mathbb{Q}_\alpha} s(\alpha)$. In conclusion, $s' \leq s$, which finishes the proof of the lemma.
\end{proof}

\begin{remark} \label{rem: interpolant restriction}
	The above construction yields the following observation: not only is $p^{(s)}$ an interpolant for $p,s, F$ and $i$, but we even have that $p^{(s)} {\restriction} \alpha$ is an interpolant for $p {\restriction} \alpha, s {\restriction} \alpha, F \cap \alpha$ and $i$ for any $\alpha < \zeta$.
\end{remark}

In the next lemma, we show that under certain conditions, the forcing $\mathbb{P}_\zeta$ admits least upper bounds of the form
\[
\bigvee_{\mathclap{\substack{s \leq q,\\ s \text{ is } (q,F,i)-\text{determined following } g}}} s.
\]
\begin{lemma} \label{lem: weakest determined}
	Let $p \in \mathbb{P}_\zeta$ and $s$ be $(p,F,i)$-determined following $g \in \prod_{\beta \in F} \kappa^{<\kappa}$. Then for every $q \leq_{F,i} p^{(s)}$ there exists an $\tilde{s} \leq q, s$ that is $(q,F,i)$-determined following $g$ such that for every $s' \leq q$, if $s'$ is $(q,F,i)$-determined following $g$, then $s' \leq \tilde{s}$. In other words, $\tilde{s}$ is the weakest $(q,F,i)$-determined condition following $g$. 
\end{lemma}

\begin{proof}
	Construct $\tilde{s}$ by induction such that for all $\alpha \leq \zeta$ we have $\tilde{s} {\restriction} \alpha \in \mathbb{P}_\alpha$, $\tilde{s} {\restriction} \alpha \leq q {\restriction} \alpha$ and $\tilde{s} {\restriction} \alpha$ is $(q {\restriction} \alpha, F \cap \alpha, i)$-determined following $g {\restriction} \alpha$.
	
	Assume $\tilde{s} {\restriction} \alpha$ has been defined; define $\tilde{s}(\alpha)$ as
	\[
	\tilde{s}(\alpha) := 
	\begin{cases}
		q(\alpha)^{[g(\alpha)]} & \text{ if } \alpha \in F \\
		q(\alpha)& \text{ otherwise.}
	\end{cases}
	\]
	If $\alpha \notin F$, there is nothing to prove. For $\alpha \in F$, observe that since $\tilde{s} {\restriction} \alpha \leq q {\restriction} \alpha$ is $(q {\restriction} \alpha, F \cap \alpha, i)$-determined following $g {\restriction} \alpha$ and $q  \leq_{F,i} p^{(s)}$, so by the above remark we can conclude $\tilde{s} {\restriction} \alpha \leq s {\restriction} \alpha$. But
    \[
        s {\restriction} \alpha \Vdash \exists \nu: g(\alpha) \in \succ_p(\nu) \wedge \nu \in \splitt_i(p(\alpha))
    \]
    and $q {\restriction} \alpha \Vdash \splitt_i(p(\alpha)) = \splitt_i(q(\alpha))$, hence $\tilde{s}(\alpha)$ is well-defined. The other two properties follow easily.
	
	If $\gamma$ is a limit, then we have $\supp(\tilde{s} {\restriction} \gamma) \subseteq \supp(q) \cup F$, hence $\tilde{s} {\restriction} \gamma \in \mathbb{P}_\gamma$ is a condition. 
	
	Knowing $\tilde{s}$ to be well-defined, one can easily see that for each $s' \leq q$ that is $(q,F,i)$-determined following $g$ we have $s' \leq \tilde{s}$.
\end{proof}

\begin{fact} \label{fact: <kappa closed F,i}
	$(\mathbb{P}_\zeta, \leq_{F,i})$ is ${<}\kappa$-closed for all $\zeta, F,i$.
\end{fact}

Let us now introduce two auxiliary ``boundedness'' properties a $\mathbb{P}_\zeta$-condition may exhibit.
\begin{definition} \label{def: bounded}
	We say a condition $p \in \mathbb{P}_\zeta$ is $(F,i)$\textit{-bounded} for $F \in [\zeta]^{<\kappa}, i < \kappa$ iff there exists a $\mu < \kappa$ with
	\[
	\forall \beta \in F:\ p {\restriction} \beta \Vdash \splitt_i(p(\beta)) \subseteq \mu^{<\mu}.
	\]
\end{definition}

\begin{fact} \label{fact: bounded open}
	If $p \in \mathbb{P}_\zeta$ is $(F,i)$-bounded and $p' \leq_{F,i} p$, then $p'$ is as well.
\end{fact}

\begin{definition}
	Let $\zeta \leq \kappa^{++}, p \in \mathbb{P}_\zeta, F \in [\zeta]^{<\kappa}$ and $i < \kappa$. Take furthermore a $D \subseteq \mathbb{P}_\zeta$ that is open dense below $p$. We say $p$ is $(D,F,i)$-\textit{complete} iff there exists a $C \subseteq \prod_{\beta \in F} \kappa^{<\kappa}, |C| < \kappa$ and a family $(s_g)_{g \in C}$ in $D$ such that 
	\begin{enumerate}[a)]
		\item $s_g$ is $(p,F,i)$-determined following $g$ for all $g \in C$		
		\item whenever $s \leq p$ is $(p,F,i)$-determined following a function $g$ and $s \in D$, then $g \in C$ and $s \leq s_g$
	\end{enumerate}
\end{definition}

\begin{fact} \label{fact: complete maximal antichain}
	If $p \in \mathbb{P}_\zeta$ is $(D,F,i)$-complete as witnessed by $(s_g)_{g \in C}$, then $(s_g)_{g \in C}$ is a maximal antichain below $p$.
\end{fact}

\begin{lemma} \label{lem: C decreasing}
	Let $p' \leq_{F,i} p$ be $\mathbb{P}_\zeta$-conditions such that $p$ is $(D,F,i)$-complete and $p'$ is $(D',F,i)$-complete. Let $(s_g)_{g \in C}$ and $(s'_g)_{g \in C'}$ witness this. Then $C' \subseteq C$. If in addition $D' \subseteq D$, then we even have $s'_g \leq s_g$ for each $g \in C'$.
\end{lemma}

\begin{proof}
	Assume that $g \in C'$ and find a $t \leq s'_{g}$ with $t \in D$ (note that $s'_g \leq p$). Then $t \leq p$ is $(p,F,i)$-determined following $g$ by Fact \ref{fact: determined p p'} and thus $g \in C$ and $t \leq s_g$ by the second requirement in the definition of completeness. If $D' \subseteq D$, we may take $t = s'_{g}$ and get $s'_{g} \leq s_{g}$.
\end{proof}

In particular we know that the set $C$ in the definition of completeness is completely determined by $p$. Complete conditions are also going to be playing a major role later in Lemma \ref{lem: second proof technical}.

Our strategy for proving Property B* for all $\mathbb{P}_\zeta, \zeta \leq \kappa^{++}$ is by the equivalence of the following four statements:
\begin{enumerate}[label=\alph*($\zeta$):, leftmargin=46pt, ref=\alph*($\zeta$)]
	\item $\mathbb{P}_\alpha$ has Property B* for each $\alpha < \zeta$. \label{equiv: a}
	\item The set of $(F,i)$-bounded conditions is $\leq_{F,i}$-dense in $\mathbb{P}_\alpha$ for all $\alpha \leq \zeta$, $F \in [\alpha]^{<\kappa}$ and $i < \kappa$. \label{equiv: b}
	\item The set of $(D,F,i)$-complete conditions is $\leq_{F,i}$-dense in $\mathbb{P}_\zeta$ for all $F,i$ and open dense $D \subseteq \mathbb{P}_\zeta$. \label{equiv: c}
	\item $\mathbb{P}_\zeta$ has Property B*. \label{equiv: d}
\end{enumerate}

The implication \ref{equiv: a} $\Rightarrow$ \ref{equiv: b} is Lemma \ref{lem: b* -> bounded}, \ref{equiv: b} $\Rightarrow$ \ref{equiv: c} is Lemma \ref{lem: bounded -> complete} and \ref{equiv: c} $\Rightarrow$ \ref{equiv: d} is Lemma \ref{lem: complete -> property b*}. Thus \ref{equiv: a} $\Rightarrow$ \ref{equiv: d} establishes an induction by which Property B* is verified for all $\mathbb{P}_\zeta$.

\begin{corollary}
	$\mathbb{P}_\zeta$ has Property B* for all $\zeta \leq \kappa^{++}$.
\end{corollary}

\begin{lemma} \label{lem: b* -> bounded}
	Let $\zeta \leq \kappa^{++}$ and assume $\mathbb{P}_{\alpha}$ has Property B* for each $\alpha < \zeta$. Then for each $\alpha \leq \zeta$, $p \in \mathbb{P}_{\alpha}, F \in [\alpha]^{<\kappa}$ and $i < \kappa$ there is a condition $q \leq_{F,i} p$ that is $(F,i)$-bounded.
\end{lemma}

\begin{proof}
	We proceed by induction on $\alpha \leq \zeta$.
	\begin{itemize}
		\item $\alpha = 1$: Trivial by the inaccessibility of $\kappa$.
		\item $\alpha \to \alpha + 1$: Let $p \in \mathbb{P}_{\alpha+1}, F \in [\alpha + 1]^{<\kappa}$ and $i< \kappa$ be given. Since $\mathbb{P}_\alpha$ is ${<}\kappa$-closed, $\kappa$ remains inaccessible in $V^{\mathbb{P}_\alpha}$. Thus 
		\[
		\Vdash_{\mathbb{P}_\alpha} \forall \beta \in F\, \exists \mu_\beta < \kappa:\ \splitt_i(p(\beta)) \subseteq \mu_\beta^{<\mu_\beta}
		\]
		and considering $\sup_{\beta \in F} \mu_\beta$ we can find a name $\dot{\mu}$ for an ordinal less than $\kappa$ with
		\[
		\Vdash_{\mathbb{P}_\alpha} \forall \beta \in F:\ \splitt_i(p(\beta)) \subseteq \dot{\mu}^{<\dot{\mu}}.
		\]
		Let now $A \subseteq \mathbb{P}_\alpha$ be a maximal antichain deciding $\dot{\mu}$; we may find a $\mathbb{P}_\alpha \ni \hat{q} \leq_{F \cap \alpha,i} p {\restriction} \alpha$ with $|A {\restriction} \hat{q}| < \kappa$. Thus 
		\[
		\hat{q} \Vdash \dot{\mu} < \mu_q
		\]
		for some $\mu_q < \kappa$ and therefore 
		\[
		\forall \beta \in F:\ \hat{q} {\restriction} \beta \Vdash \splitt_i(p(\beta)) \subseteq \mu_q^{<\mu_q}.
		\]
		Setting $q := \hat{q} ^\frown p(\alpha)$ and noting that since $\hat{q} \leq_{F \cap \alpha, i} p {\restriction} \alpha$ we have $q {\restriction} \beta \Vdash \splitt(q(\beta)) = \splitt(p(\beta))$ for all $\beta \in F$, so it follows that $q$ is $(F,i)$-bounded. 
		\item $\gamma \leq \zeta$ is a limit: Let $p \in \mathbb{P}_{\gamma}, F \in [\gamma]^{<\kappa}$ and $i< \kappa$ be given. Using ${<}\kappa$-closure of $(\mathbb{P}_\gamma, \leq_{F,i})$ (see Fact \ref{fact: <kappa closed F,i}) and the inductive assumption, we can construct a $\leq_{F,i}$-decreasing sequence $(q_\beta)_{\beta \in F}$ in $\mathbb{P}_\gamma$ with the following properties:
		\begin{itemize}
			\item $\forall \beta \in F\, \forall \beta' \in F \cap \beta:\ q_\beta \leq_{F,i} q_{\beta'} \leq_{F,i} p$
			\item $\forall \beta \in F\, \exists \mu_\beta < \kappa\, \forall \beta' \in F \cap (\beta +1):\ q_\beta {\restriction} \beta' \Vdash_{\mathbb{P}_{\beta'}} \splitt_i(q_\beta(\beta')) \subseteq \mu_\beta^{<\mu_\beta}.$
		\end{itemize}
		Again using ${<}\kappa$-closure of $(\mathbb{P}_\gamma, \leq_{F,i})$, set $q$ to a $\leq_{F,i}$-lower bound of $(q_\beta)_{\beta \in F}$ and $\mu := \sup_{\beta \in F} \mu_\beta$. Now $q \leq_{F,i} p$ and
		\[
		\forall \beta \in F:\ q {\restriction} \beta \Vdash \splitt_i(q(\beta)) \subseteq \mu^{<\mu}. \qedhere
		\]
	\end{itemize}
\end{proof}

\begin{lemma} \label{lem: bounded -> complete}
	Let $\zeta \leq \kappa^{++}, F \in [\zeta]^{<\kappa}, i < \kappa$ and suppose $p \in \mathbb{P}_{\zeta}$ is $(F,i)$-bounded. Let furthermore $D \subseteq \mathbb{P}_{\zeta}$ be open dense below $p$. Then there is a $q \leq_{F,i} p$ which is $(D,F,i)$-complete. 
	
	In particular, if the set of $(F,i)$-bounded conditions is $\leq_{F,i}$-dense in $\mathbb{P}_\zeta$, then for all open dense $D \subseteq \mathbb{P}_\zeta$ the set of $(D,F,i)$-complete conditions is $\leq_{F,i}$-dense as well.
\end{lemma}

\begin{proof}
	By assumption $p$ is $(F,i)$-bounded, hence we can find a $\mu$ such that
	\[
	\forall \beta \in F:\ p {\restriction} \beta \Vdash \splitt_i(p(\beta)) \subseteq \mu^{<\mu}.
	\]
	Our strategy is to consider all possible choices a $(p,F,i)$-determined condition might make at the $i$-th splitting front of $p$ and then interpolate on the witnesses of such choices. 
	Since we have a uniform bound $\mu$ on the respective splitting fronts, this will require us to only iterate through ${<}\kappa$ many possibilities. Set $\tilde{\mu}_\beta := \sup_{j \leq \mu} f_\beta(j)$ and consider the set
	\[
	\tilde{C} := \prod_{\beta \in F} \tilde{\mu}_\beta^{\leq \mu}.
	\]
	Whenever $s$ is $(p,F,i)$-determined following some $g$, then $g \in \tilde{C}$. Enumerate $\tilde{C}$ as $(g_{j+1})_{j < \delta}$ with $\delta < \kappa$. We now construct a $\leq_{F,i}$-decreasing sequence $(t_j)_{j < \delta}$:
	\begin{itemize}
		\item $j = 0$: Set $t_0 := p$.
		\item $j \to j+1$: If there exists an $s \leq t_j$ that is $(p,F,i)$-determined following $g_{j+1}$, pick an arbitrary such condition from $D$ (this is possible, since $D$ is dense below $p$) and call it $\tilde{s}_{g_{j+1}}$. Set $t_{j+1} := t_j^{(\tilde{s}_{g_{j+1}})}$. If there is no such $s$, simply set $t_{j+1} := t_j$. In any case we have $t_{j+1} \leq_{F,i} t_j$.
		\item $\lambda$ is a limit: Set $t_\lambda$ to a $\leq_{F,i}$-lower bound of $(t_j)_{j < \lambda}$ (see Fact \ref{fact: <kappa closed F,i}).
	\end{itemize}
	
	Set $q$ to a $\leq_{F,i}$-lower bound of $(t_j)_{j < \delta}$. We know $q \leq_{F,i} p$. Now let
	\[
	C := \left\{g \in \tilde{C}:\ \tilde{s}_g \text{ exists} \right\},
	\]
	i.e. $C$ is the set of all $g_{j+1}$ for which a witness was found in the inductive step $j \to j+1$. We have $|C| < \kappa$. Finally, for each $g = g_{j+1} \in C$ apply Lemma \ref{lem: weakest determined} to $p = t_j$, $s = \tilde{s}_{g_{j+1}}$ and $q = q$ to construct the condition $s_g$. We have $s_g \in D$ since $s_g \leq \tilde{s}_g \in D$ and $D$ is open. 
	
	We verify that $q$ is $(D,F,i)$-complete, witnessed by $(s_g)_{g \in C}$. The first condition in the definition of completeness follows by construction. The second follows immediately from Lemma \ref{lem: weakest determined} by noting that if $s \leq q$ is $(q,F,i)$-determined following $g$, then $g = g_{j+1}$ for some $j < \delta$, and thus a witness was found in the inductive step $j \to j+1$; hence $g \in C$.
\end{proof}

\begin{lemma} \label{lem: complete -> property b*}
	If the set of $(D,F,i)$-complete conditions is $\leq_{F,i}$-dense in $\mathbb{P}_\zeta$ for all $F,i$ and $D \subseteq \mathbb{P}_\zeta$ open dense, then $\mathbb{P}_\zeta$ has Property B*.
\end{lemma}
\begin{proof}
	We have seen in Lemma \ref{lem: iterated fusion limits exist} that the fusion condition is always fulfilled. We will now prove that $\mathbb{P}_{\zeta}$ fulfils the antichain condition: let $A \subseteq \mathbb{P}_\zeta$ be a maximal antichain, $p \in \mathbb{P}_\zeta, F \in [\zeta]^{<\kappa}$ and $i < \kappa$. Find a $q \leq_{F,i} p$ that is $(D,F,i)$-complete, where
	\[
	D = \{r \in \mathbb{P}_\zeta:\ |A {\restriction} r| = 1\}
	\]
	and let $(s_g)_{g \in C}$ witness this. Since $(s_g)_{g \in C}$ is a maximal antichain below $q$ by Fact \ref{fact: complete maximal antichain}, it is easy to see that 
	\[
	A {\restriction} q \subseteq \{r \in A:\ \exists g \in C:\ A {\restriction} s_g = \{r\}\}
	\]
	and thus $|A {\restriction} q| \leq |C| < \kappa$.
\end{proof}

From this point onward, assume that 
\[
V \models |2^\kappa| = \kappa^+.
\]
From among our stated goals at the beginning of this section, only one remains to be verified; our interest now turns to the $\kappa^{++}$-chain condition:

\begin{theorem} \label{th: cc}
	$\mathbb{P}$ has the $\kappa^{++}$-c.c.
\end{theorem}

This will follow easily from Lemma \ref{lem: jech cc} once we have proven that each $\mathbb{P}_\alpha$ for $\alpha < \kappa^{++}$ has a dense subset of size $\kappa^{+}$.

For the purposes of the next definition, for each $\alpha < \kappa^{++}$ fix a $\mathbb{P}_\alpha$-name $\dot{c}_\alpha$ for a bijection $c_\alpha : (PT_{f_\alpha})^{V^{\mathbb{P_\alpha}}} \to (\mathcal{P}(\kappa))^{V^{\mathbb{P_\alpha}}}$ such that $c_\alpha(\mathbbm{1}_{PT_{f_\alpha}}) = \emptyset$. Let us also fix a $\theta$ sufficiently large (e.g. $\theta > |2^\mathbb{P}|$) and a well-ordering $\preceq$ of $H(\theta)$ (by which we mean the sets hereditarily of size ${<} \theta$, not $H_\theta$ as defined below).

\begin{definition}
	Let $\alpha < \kappa^{++}$. 
	\begin{itemize}		
		\item A $\mathbb{P}_\alpha$-name $\dot{\tau}$ for a subset of $\kappa$ is $\alpha$-good iff $\dot{\tau}$ is a nice name of the form \[\dot{\tau} = \{\{j\} \times A_j: j < \kappa\},\] where $A_j \subseteq H_\alpha$ and $|A_j| \leq \kappa$ for all $j < \kappa$.
		\item A condition $p \in \mathbb{P}_\alpha$ is in $\tilde{H}_\alpha$ iff $p {\restriction} \beta \in H_\beta$ for each $\beta< \alpha$ and, if $\alpha = \beta + 1$ is a successor, there additionally is a $\beta$-good name $\dot{\tau}$ such that $p {\restriction} \beta \Vdash_{\beta} \dot{c}_{\beta}(p(\beta)) = \dot{\tau}$ and $p(\beta)$ is the $\preceq$-least $\mathbb{P}_\beta$-name that satisfies this relation for $p {\restriction} \beta$ and $\dot{\tau}$. Now let $H_\alpha$ be the closure of $\tilde{H}_\alpha$ under canonical fusion limits (see Remark \ref{rem: weakest fusion limit}), where we always choose the pointwise $\preceq$-least name for the canonical fusion limit (i.e. for every $\beta < \alpha$, $q(\beta)$ is the $\preceq$-least name for the $\beta$-th entry of the canonical fusion limit $q$ of a fusion sequence). \footnote{For $\alpha = 0$ let $H_0 = P_0$ be the trivial forcing notion.}
	\end{itemize}
\end{definition}

\begin{remark} \label{rem: restrictions Ha}
    It is easy to see that for $p \in H_\alpha$ (not only $p \in \tilde{H}_\alpha$) we have $p {\restriction} \beta \in H_\beta$ for all $\beta < \alpha$ as well. 
    
    Note also that there is a canonical embedding $H_\beta \hookrightarrow H_{\alpha}$ for $\beta < \alpha$.
\end{remark}

\begin{remark}
	$H_\alpha$-conditions and $\alpha$-good names appeared first as $H_\kappa$-$\mathbb{P}_\alpha$-names in \cite{baumhauer} and are themselves a straightforward generalization of hereditarily countable names as introduced in \cite{shelah_proper}.
\end{remark}

\begin{lemma} \label{lem: ha dense}
	For every $0 < \alpha < \kappa^{++}$, $F \in [\alpha]^{<\kappa}$ and $i < \kappa$, $H_\alpha$ is $\leq_{F,i}$-dense in $\mathbb{P}_\alpha$ and $|H_\alpha| = \kappa^+$.
\end{lemma}
\begin{proof}
	We prove the statements by induction on $\alpha$.
	\begin{itemize}
		\item $\alpha=1$: We have $H_1 = \mathbb{P}_1$ and $|\mathbb{P}_1| = |PT_{f_0}| = \kappa^{+}$.
		\item $\alpha \to \alpha + 1$: Let $p \in \mathbb{P}_{\alpha+1}$, $F \in [\alpha+1]^{<\kappa}$ and $i < \kappa$. Using the inductive hypothesis, we may assume $p {\restriction} \alpha \Vdash_\alpha \dot{c}_\alpha(p(\alpha)) = \{\{j\} \times A_j: j < \kappa\}$ with $A_j \subseteq H_\alpha$ for all $j < \kappa$. Additionally using Property B*, construct a fusion sequence $\langle q_j, F_j: j < \kappa \rangle$ with
		\begin{itemize}
			\item $\forall j < \kappa: q_j \in H_\alpha$ and $|A_j {\restriction} q_j| < \kappa$,
			\item $q_0 \leq_{F \cap \alpha, i} p {\restriction} \alpha$,
			\item $\forall j < \ell < \kappa: q_\ell \leq_{F_j, i + j} q_j$ and $F \cap \alpha \subseteq F_j$,
		\end{itemize}
		where the $F_j$ are constructed using a bookkeeping argument. Let $q_\kappa$ be a canonical fusion limit of this sequence that is a member of $H_\alpha$. By the first property of the fusion sequence, \[\dot{\tau} = \{\{j\} \times (A_j {\restriction} q_\kappa): j < \kappa\}\] is an $\alpha$-good name and $q_\kappa \Vdash_\alpha \dot{c}_\alpha(p(\alpha)) = \dot{\tau}$. Let thus $\dot{r}$ be the $\preceq$-least $\mathbb{P}_\alpha$-name that satisfies $q_\kappa \Vdash_\alpha \dot{c}_\alpha(\dot{r}) = \dot{\tau}$; now we can (pedantically, using Remark \ref{rem: restrictions Ha}) conclude $H_{\alpha + 1} \ni (q_\kappa\, ^\frown \dot{r}) \leq_{F,i} p$.
		
		Since $|H_\alpha| = \kappa^+$ and there are only $|(\kappa^+)^\kappa|=\kappa^+$ many $\alpha$-good names for reals, we get $|\tilde{H}_{\alpha+1}| = \kappa^+$ and therefore also $|H_{\alpha+1}| = \kappa^+$ by standard arguments.
		
		\item $\gamma$ is a limit: If $\cf(\gamma) = \kappa^+$, density is trivial and $|H_\gamma| \leq |\bigcup_{\beta < \gamma} H_\beta| \leq \kappa^+$. 
		
		Assume $\cf(\gamma) = \delta \leq \kappa$ and let furthermore $p \in \mathbb{P}_{\gamma}$ , $F \in [\gamma]^{<\kappa}$ and $i < \kappa$ be given. For a cofinal sequence $(\beta_j)_{j < \delta}$ construct a fusion sequence $\langle q_j, F_j: j < \delta \rangle$ with
		\begin{itemize}
            \item $\forall j < \delta: F \cap \beta_j \subseteq F_j \subseteq \beta_j$,
			\item $\forall j < \delta: q_j {\restriction} \beta_j \leq_{F_j, i + j} p {\restriction} \beta_j$,
			\item $\forall j < \ell < \delta: q_\ell \leq_{F_j, i+j} q_j$,
   			\item $\forall j < \delta : q_j \in H_\gamma$, which may be achieved by having $q_j {\restriction} \beta_j \in H_{\beta_j}$ and letting $q_j(\beta)$ be the trivial condition for $\beta \geq \beta_j$.
		\end{itemize}
		The $F_j$ are again constructed using a bookkeeping argument. Set $q_\delta$ to be a fusion limit contained in $H_\gamma$; then we have $H_\gamma \ni q_\delta \leq_{F,i} p$. Lastly, using Remark \ref{rem: restrictions Ha}, we get $|H_\gamma| \leq \left| \prod_{j < \delta} H_{\beta_j} \right| \leq \kappa^+$. \qedhere
	\end{itemize}  
\end{proof}

\begin{lemma} \label{lem: jech cc}
	Let $\langle \mathcal{P}_\alpha, \dot{\mathcal{Q}}_\beta : \alpha \leq \zeta, \beta < \zeta \rangle$ be an iteration such that 
	\[
	\forall \alpha < \zeta:\ \mathcal{P}_\alpha \text{ has the } \theta\text{-c.c.},
	\] 
	where $\theta$ is a regular uncountable cardinal. If $\mathcal{P}_\zeta$ is a direct limit and, additionally, either $\cf(\zeta) \neq \theta$ or the set $\{\gamma < \zeta:\ \mathcal{P}_\gamma \text{ is a direct limit}\}$ is stationary, then $\mathcal{P}_\zeta$ has the $\theta$-c.c.
\end{lemma}
\begin{proof}
	See \cite[Theorem~16.30]{jech}.
\end{proof}

\begin{proof}[Proof of Theorem \ref{th: cc}]
	By Lemma \ref{lem: ha dense}, each $\mathbb{P}_\alpha$ has a dense subset of size ${\leq}\kappa^{+}$ and therefore satisfies the $\kappa^{++}$-c.c.; our desired conclusion thus follows easily from Lemma \ref{lem: jech cc} and by noting that the set $\{\gamma < \kappa^{++}: \cf(\gamma) = \kappa^{+}\}$ is stationary in $\kappa^{++}$.  
\end{proof}

As we have remarked at the beginning of this section, we get the following corollary:

\begin{corollary}
	Forcing with $\mathbb{P}_\alpha, \alpha \leq \kappa^{++}$ does not collapse cardinals.
\end{corollary}

\begin{lemma} \label{lem: no collapse}
	We have
	\begin{itemize}
		\item If $\alpha < \kappa^{++}$, then $V^{\mathbb{P}_\alpha} \models  |2^\kappa| = \kappa^+$.
		\item If $\cof(\alpha) > \kappa$, then $V^{\mathbb{P}_\alpha} \models 2^\kappa = \bigcup_{\beta < \alpha} (2^\kappa \cap V^{\mathbb{P}_\beta})$.
		\item $V^{\mathbb{P}} \models  |2^\kappa| = \kappa^{++}$.
	\end{itemize}
\end{lemma}

\begin{proof}		
	Suppose $\alpha < \kappa^{++}$. Let $\dot{\tau}$ be a $\mathbb{P}_\alpha$-name and $p \in \mathbb{P}_\alpha$ force $\dot{\tau}$ to be a subset of $\kappa$. Without loss of generality assume $\dot{\tau} = \{\{j\} \times A_j: j < \kappa\}$ is a nice name with $A_j \subseteq H_\alpha$ for all $j < \kappa$. Just like in Lemma \ref{lem: ha dense}, construct a fusion sequence $\langle q_j, F_j: j < \kappa \rangle$ below $p$ with $|A_j {\restriction} q_j| < \kappa$ for all $j < \kappa$. The fusion limit $q_\kappa$ forces $\dot{\tau}$ to be equal to an $\alpha$-good name, of which there are only $\kappa^+$ many. If we additionally assume $\cf(\alpha) > \kappa$, then $q_\kappa$ forces $\dot{\tau}$ to be equal to a $\mathbb{P}_\gamma$-name for some $\gamma < \alpha$. The first two statements thus follow by a density argument.
	
	The last point follows immediately from the previous two.
\end{proof}

For $\alpha < \kappa^{++}$ we can define in $V^{\mathbb{P}_\alpha}$ the ${\leq}\kappa$-support tail iteration
$\langle \tilde{\mathbb{P}}_\gamma, \dot{\tilde{\mathbb{Q}}}_\beta : \gamma \leq \kappa^{++}, \beta < \kappa^{++} \rangle$ such that $\Vdash_{\tilde{\mathbb{P}}_\gamma} \dot{\tilde{\mathbb{Q}}}_\gamma = \dot{\mathbb{Q}}_{\alpha + \gamma}$. Set $\mathbb{P}_{\alpha, \kappa^{++}} := \tilde{\mathbb{P}}_{\kappa^{++}}$ and note that $\mathbb{P}_{\alpha, \kappa^{++}}$ has Property B* as well. It follows from standard proper forcing arguments (adapting \cite[Theorem III.3.4]{shelah_proper} to the case of $\kappa$-properness) that $\mathbb{P} \simeq \mathbb{P}_\alpha \star \mathbb{P}/_{\dot{G}_\alpha}$ is densely embedded in $\mathbb{P}_\alpha \star \mathbb{P}_{\alpha, \kappa^{++}}$.

\section{First Proof}

We are now equipped to present the first proof of the relative consistency of
\[
\text{ZFC} + |2^\kappa| = \kappa^{++} + \mathcal{SN} = [2^\kappa]^{\leq \kappa^+}.
\]
Starting with a model of $|2^\kappa| = \kappa^+$, we consider a ${\leq}\kappa$-supported forcing iteration $\langle \mathbb{P}_\alpha, \dot{\mathbb{Q}}_\beta : \alpha \leq \kappa^{++}, \beta < \kappa^{++} \rangle$ with 
\[
\forall \alpha < \kappa^{++}:\ \Vdash_{\mathbb{P}_\alpha} \dot{\mathbb{Q}}_\alpha = PT_{f_\alpha},
\]
such that $(f_\alpha)_{\alpha < \kappa^{++}} \in V$ and each increasing $f \in \kappa^\kappa \cap V$ appears as an $f_\alpha$ cofinally often. Set $\mathbb{P} := \mathbb{P}_{\kappa^{++}}$. By Lemma \ref{lem: no collapse} we see $V^\mathbb{P} \models |2^\kappa| = \kappa^{++}$.

By a density argument, the $\alpha$-th generic real $\dot{s}_\alpha$ will encode a covering of the ground model reals satisfying the `challenge' $f_\alpha$. For this argument it is sufficient that only $f_\alpha$ from some dominating family appear cofinally often; from the perspective of some intermediate model $V^{\mathbb{P}_\alpha}$, the tail forcing $\mathbb{P}_{\alpha, \kappa^{++}}$ fulfils this criterion. Hence the observation can be extended to the set of reals appearing already in some $V^{\mathbb{P}_\alpha}$; the following theorem formalizes this.

\begin{theorem}
	$V^\mathbb{P} \models \forall \alpha < \kappa^{++}: 2^\kappa \cap V^{\mathbb{P}_\alpha} \in \mathcal{SN}.$ \label{th: easy inclusion aux}
\end{theorem}
\begin{proof}
	Working within $V^\mathbb{P}$, take $\alpha < \kappa^{++}$ and $f \in \kappa^\kappa$. Since $\mathbb{P}$ is $\kappa^\kappa$-bounding by Fact \ref{fact: b* implies proper and bounding}, we find an $h \in \kappa^\kappa \cap V$ with $f \leq h$ and $\beta > \alpha$ with $f_\beta(i) = |2^{h(i)}|$ for all $i < \kappa$. In $V$ we may construct bijections $c_\ell : |2^\ell| \to 2^\ell$ for $\ell < \kappa$.
	
	Working now in $V^{\mathbb{P}_\beta}$, define the function $\dot{\sigma}(i) = c_{h(i)}(\dot{s}_\beta(i))$. For $x \in 2^\kappa \cap V^{\mathbb{P}_\alpha}$ the set
	\[
	D_x := \{p \in \mathbb{Q}_\beta:\ \exists i < \kappa: p \Vdash \dot{\sigma}(i) = x {\restriction} h(i)\}
	\]
	is dense; in fact, it is easy to see that for any $p \in \mathbb{Q}_\beta$ and $\eta \in \splitt(p), j = \dom(\eta)$ we have $p^{[\eta ^\frown c^{-1}_{h(j)}(x {\restriction} h(j)) ]} \in D_x$. Here $c^{-1}_{h(j)}(x {\restriction} h(j))$ is well-defined, since $2^{<\kappa} \cap V = 2^{<\kappa} \cap V^{\mathbb{P}_\beta}$. Hence $(\sigma(i))_{i < \kappa}$ provides the required covering for the challenge $f$ and $2^\kappa \cap V^{\mathbb{P}_\alpha} \in \mathcal{SN}$ follows.
\end{proof}

If $V^{\mathbb{P}} \models X \subseteq 2^\kappa, |X| \leq \kappa^+$, then by the $\kappa^{++}$-c.c., $X$ already appears at some intermediate stage $V^{\mathbb{P}_\alpha}$. We thus get one direction of our desired result by the previous theorem.
\begin{theorem} \label{th: easy inclusion}
	$V^{\mathbb{P}} \models [2^\kappa]^{\leq \kappa^+} \subseteq \mathcal{SN}$.
\end{theorem} 

In order to lift the arguments appearing in \cite{smz_no_cohen}, we require additional large cardinal assumptions on $\kappa$. A priori it is sufficient for our purposes for $\kappa$ to merely be weakly compact, since the only occasion at which a property stronger than inaccessibility is utilized is a crucial invocation of the tree property in Lemma \ref{lem: initial segment}. However, the aforementioned lemma is invoked not only in $V$, but also at intermediate stages $V^{\mathbb{P}_\alpha}$; it might be the case that weak compactness of $\kappa$ is by that point destroyed.

The following large cardinal property was introduced by Villaveces \cite[Definition 4]{villaveces}:

\begin{definition}
	Let $\theta$ be an ordinal. We say an inaccessible cardinal $\kappa$ is $\theta$-\textit{strongly unfoldable} iff for all transitive models $M$ of ZF$^-$ (ZF without the Power Set Axiom) such that $|M| = \kappa, \kappa \in M$ and ${}^{<\kappa}M \subseteq M$ there exists a transitive model $N$ with $V_{\theta} \cup \{\theta\} \subseteq N$ and an elementary embedding $j : M \to N$ with critical point $\kappa$ and $j(\kappa) \geq \theta$.
	
	Furthermore, call $\kappa$ strongly unfoldable iff it is $\theta$-strongly unfoldable for all $\theta$. 
\end{definition}

Strongly unfoldable cardinals are weakly compact and are downwards absolute to $L$ \cite{villaveces}. Villaveces also observed that Ramsey cardinals are strongly unfoldable in $L$ (though they may fail to be such in $V$). The consistency strength of a strongly unfoldable cardinal thus slots between a weakly compact and Ramsey cardinal, with it being a conservative enough strengthening of weak compactness as to still be consistent with $V=L$.

Of interest to us is a preservation theorem by Johnstone \cite{johnstone_unfoldable}.

\begin{theorem}[Johnstone \cite{johnstone_unfoldable}]
	For any $\kappa$ strongly unfoldable there is a forcing extension in which the strong unfoldability of $\kappa$ is indestructible under ${<}\kappa$-closed, $\kappa$-proper forcing notions.
\end{theorem}

We stress that the full strength of strong unfoldability is not used in our proof; we merely require it in order to make the weak compactness of $\kappa$ indestructible by the forcings $\mathbb{P}_\alpha$. 

For a strongly unfoldable $\kappa$, after forcing indestructibility using Johnstone's theorem, we may collapse a potentially blown up $2^\kappa$ back to $\kappa^+$ with a ${<}\kappa^+$-closed forcing \footnote{${<}\kappa^+$-closed forcings and two-step iterations of $\kappa$-proper forcings are $\kappa$-proper.}. Throughout this section we may therefore assume
\begin{align*}
	V \models `` |2^\kappa| = \kappa^+ + &\text{ the strong unfoldability of $\kappa$ is indestructible} \\
	&\text{under ${<}\kappa$-closed, $\kappa$-proper forcing notions''.}
\end{align*}

We now set out to prove $V^\mathbb{P} \models \mathcal{SN} \subseteq [2^\kappa]^{\leq \kappa^+}$.

The statement of the next two lemmas takes place in $V^{\mathbb{P}_\alpha}$. Recall that $\mathbb{P}_{\alpha, \kappa^{++}}$ denotes the tail forcing.
\begin{lemma} \label{lem: initial segment}
	Let $\alpha < \kappa^{++}$ be an ordinal, $\dot \tau$ a $\mathbb{P}_{\alpha,\kappa^{++}}$-name for a real in $2^\kappa$, $F \in [\kappa^{++} \backslash \alpha]^{< \kappa}$ and $i < \kappa$. Assume furthermore that $p \in \mathbb{P}_{\alpha, \kappa^{++}}$ forces $\dot \tau \notin V^{\mathbb{P}_\alpha}$. Then there exists a $\delta <\kappa$ such that 
	\[
	\forall  \eta \in 2^\delta \, \exists q \leq _{F,i} p:\ q \Vdash_{\mathbb{P}_{\alpha, \kappa^{++}}} \eta \ntriangleleft \dot \tau.
	\]
	We will write $\delta_{p, F,i}$ for the least such $\delta$.
\end{lemma}

\begin{proof}
	Suppose not. Then we can find $\alpha, \dot{\tau}, F,i$ and $p$ with
	\[
	\forall \delta < \kappa\, \exists \eta_\delta \in 2^\delta:\ \neg(\exists q \leq_{F,i} p:\ q \Vdash \eta \ntriangleleft \dot{\tau}).
	\]
	Set $T:= \{\eta_\delta {\restriction} \ell: \ell \leq \delta < \kappa\}$. By virtue of the preparation of $\kappa$, 
	\[
	V^{\mathbb{P}_\alpha} \models \kappa \text{ is weakly compact}
	\]
	and therefore, since $T$ is a ${<}\kappa$-splitting tree of height $\kappa$, it has a branch $b^*$ in $V^{\mathbb{P}_\alpha}$. Since $p$ forces $\dot{\tau} \notin V^{\mathbb{P}_\alpha}$, there is a $\mathbb{P}_{\alpha, \kappa^{++}}$-name $\dot{\ell}$ for an ordinal less than $\kappa$ such that $p \Vdash \dot{\tau} {\restriction} \dot{\ell} \neq b^* {\restriction} \dot{\ell}$. As $\mathbb{P}_{\alpha, \kappa^{++}}$ satisfies Property B*, there is a $q \leq_{F,i} p$ and $\ell^* < \kappa$ with $q \Vdash \dot{\ell} < \ell^*$.
	
	Since $b^* {\restriction} \ell^* \in T$, there is a $\delta \geq \ell^*$ such that $b^* {\restriction} \ell^* = \eta_\delta {\restriction} \ell^*$. But this means $q \Vdash \dot{\tau} {\restriction} \ell^* \neq \check{\eta}_\delta {\restriction} \ell^*$ and therefore $q \Vdash \eta_\delta \ntriangleleft \dot{\tau}$, a contradiction.
\end{proof}

In the following we refer to  pointwise (everywhere) domination $\leq$ and not just the eventually dominating relation. For a ${<}\kappa$-closed, $\kappa^\kappa$-bounding forcing, the ground model $\kappa$-reals form a pointwise dominating family.

\begin{definition}
	Let $D\subseteq \kappa^\kappa$ be a dominating family. We say that $H$ has index $D$ iff $H=\{h_f \colon f\in D\}$ and $\forall i <\kappa \colon h_f (i) \in 2^{f(i)}$.
\end{definition}

\begin{fact} \label{fact: sn iff dominating}
	\begin{align*}
		X \in \mathcal{SN} &\Leftrightarrow \forall D \,\text{dominating} \,\,\, \exists H  \, \text{with index} \, D \colon X \subseteq \bigcap_{f \in D} \bigcup_{\alpha< \kappa} [h_f(\alpha)].
	\end{align*}
\end{fact}

\begin{lemma} \label{lem: condition tree}
	Let $D \in V$ be a dominating family, $\alpha < \kappa^{++}$ and $H \in V^{\mathbb{P}_\alpha}$ have index $D$. Let furthermore $\dot{\tau}$ be a name for an element of $2^\kappa$ with $\Vdash_{\mathbb{P}_{\alpha, \kappa^{++}}} \dot{\tau} \notin V^{\mathbb{P}_\alpha}$. Then we have
	\[
	\Vdash_{\mathbb{P}_{\alpha, \kappa^{++}}} \dot{\tau} \notin \bigcap_{f\in D} \bigcup_{i <\kappa} [h_f(i)].
	\]
\end{lemma}

\begin{proof}
	We prove the claim with a density argument, let therefore $p \in \mathbb{P}_{\alpha, \kappa^{++}}$ be arbitrary. Working in $V^{\mathbb{P}_\alpha}$, we will construct an increasing sequence $(\delta_i)_{i < \kappa}$ of ordinals less than $\kappa$. On the tree 
	\[
	T := \{g \in \prod_{j < i} 2^{\delta_j}: i < \kappa\}
	\]
	we shall construct a mapping $\mathfrak{q}: T \to \mathbb{P}_{\alpha, \kappa^{++}}$ and a sequence of increasing sets $(F_i)_{i < \kappa}$ such that whenever $b \in \prod_{j < \kappa} 2^{\delta_j}$ is a branch of $T$ in $V^{\mathbb{P}_\alpha}$, then
	\[
	\langle \mathfrak{q}(b {\restriction} i), F_i:\ i < \kappa \rangle
	\]
	is a fusion sequence below $p$. Each condition $\mathfrak{q}(g)$ will carry some information about an increasingly long initial segment of $\dot{\tau}$. More specifically, we will ensure that for all $i < \kappa$ and $g \in \prod_{j < i} 2^{\delta_j}$ we have
	\[
	\mathfrak{q}(g) \Vdash \forall j < i:\ g(j) \ntriangleleft \dot{\tau}.
	\]
	
	We define $\mathfrak{q}(g)$ for $g \in \prod_{j < i} 2^{\delta_j}$ by induction in $i$.
	
	\begin{itemize}
		\item $i=0$: Set $\mathfrak{q}(\emptyset) := p$ and $F_0 := \emptyset$.
		\item $i \to i+1$: Assume $\mathfrak{q}(g)$ is defined for every $g \in \prod_{j < i} 2^{\delta_j}$. Using Lemma \ref{lem: initial segment} we can define $\delta_{i} := \sup\Big(\{\delta_{\mathfrak{q}(g), F_i, i}: g \in \prod_{j < i} 2^{\delta_j}\} \cup \{\delta_j + 1: j < i\} \Big)$ and for every $g \in \prod_{j < i} 2^{\delta_j}, \eta_{i} \in 2^{\delta_{i}}$ find a condition $\mathfrak{q}(g ^\frown \eta_{i}) \leq_{F_i,i} \mathfrak{q}(g)$ with
		\[
		\mathfrak{q}(g ^\frown \eta_{i}) \Vdash \eta_{i} \ntriangleleft \dot{\tau}.
		\]
		Now since $\mathfrak{q}(g ^\frown \eta_{i}) \leq \mathfrak{q}(g)$, we also have
		\[
		\mathfrak{q}(g ^\frown \eta_{i}) \Vdash \forall j < i:\ g(j) \ntriangleleft \dot{\tau}.
		\]
		Use a bookkeeping argument to define $F_{i+1}$.
		\item  $\lambda < \kappa$ is a limit: By construction, for every $g \in \prod_{j < \lambda} 2^{\delta_j}$ the sequence $(\mathfrak{q}(g {\restriction} j))_{j < \lambda}$ is a fusion sequence. Set $\mathfrak{q}(g)$ to be a fusion limit of said sequence and $F_\lambda := \bigcup_{j < \lambda} F_j$. Note that we have
		\[
		\mathfrak{q}(g) \Vdash \forall j < \lambda:\ g(j) \ntriangleleft \dot{\tau}.
		\]
	\end{itemize}
	
	This concludes the construction of $\mathfrak{q}$. Let now $f \in D$ dominate the function $i \mapsto \delta_i$ and set $\eta_i := h_f(i) {\restriction} \delta_i$. Now $(\mathfrak{q}(g_i))_{i < \kappa}$ with $g_i = (\eta_j)_{j < i}$ is a fusion sequence and has a fusion limit $q_\kappa$. It follows that
	\[
	q_\kappa \Vdash \eta_i \ntriangleleft \dot{\tau}
	\]
	for each $i < \kappa$ and therefore $q_\kappa \Vdash \dot{\tau} \notin \bigcap_{f\in D} \bigcup_{i <\kappa} [h_f(i)]$. Thus the set of conditions that force $\dot{\tau} \notin \bigcap_{f\in D} \bigcup_{i <\kappa} [h_f(i)]$ is dense in $\mathbb{P}_{\alpha, \kappa^{++}}$.
\end{proof}

We see that every intermediate model $V^{\mathbb{P}_\alpha}$ believes that a set $X \subseteq 2^\kappa$ which contains a real appearing in a later model will never be strong measure zero with respect to any test conducted in $V^{\mathbb{P}_\alpha}$. This essentially gives us our theorem.

\begin{theorem}
	$V^\mathbb{P} \models \mathcal{SN} = [2^\kappa]^{\leq \kappa^+}$.
\end{theorem}

\begin{proof}
	Since we already saw one inclusion in Theorem \ref{th: easy inclusion}, it suffices to show $V^\mathbb{P} \models \mathcal{SN} \subseteq [2^\kappa]^{\leq \kappa^+}$. Let now $X \in V^\mathbb{P}$ be of size $\kappa^{++}$ and $D$ be a dominating family in $V^\mathbb{P}$ which lies in $V$. We will show that there is no $H \in V^\mathbb{P}$ with index $D$ such that 
	\[
	X \subseteq \bigcap_{f \in D} \bigcup_{i < \kappa} [h_f(i)],
	\]
	hence $X$ is not strong measure zero by Fact \ref{fact: sn iff dominating}. Towards a contradiction, assume such an $H$ exists. Note that since $D$ is in $V$, the set $H$ can have cardinality at most $\kappa^+$; since $\mathbb{P}$ fulfils the $\kappa^{++}$-c.c., we know $H$ must already appear in some intermediate model $V^{\mathbb{P}_\alpha}$. But $|X| = \kappa^{++}$, thus there must be an $x \in X$ with $x \notin V^{\mathbb{P}_\alpha}$. 
	
	Working in $V^{\mathbb{P_\alpha}}$, let $\dot{x}$ and $\dot{X}$ be $\mathbb{P}_{\alpha, \kappa^{++}}$-names for $x$ and $X$, respectively, so that we have
	\[
	\Vdash_{\mathbb{P}_{\alpha, \kappa^{++}}} \dot{x} \in \dot{X} \wedge \dot{x} \notin V^{\mathbb{P}_\alpha}.
	\] 
	Then by Lemma \ref{lem: condition tree} we have
	\[
	\Vdash_{\mathbb{P}_{\alpha, \kappa^{++}}} \dot{x} \notin \bigcap_{f\in D} \bigcup_{i <\kappa} [h_f(i)],
	\]
	and thus $X \nsubseteq \bigcap_{f \in D} \bigcup_{i < \kappa} [h_f(i)]$ in $V^\mathbb{P}$, a contradiction.
\end{proof}

\section{Coding of Continuous Functions} \label{sec: coding}

For the reader's convenience we collect some selected facts about the coding of continuous functions that are going to find use in the next section.

Throughout this section, every tree $T$ is assumed to be a tree on $2^{<\kappa}$.

\begin{definition}
	Let $T$ be a tree and  $(T_\eta)_{\eta \in 2^{<\kappa}}$ a family of trees. Then $\langle T, (T_\eta)_{\eta \in 2^{<\kappa}} \rangle$ is a \textit{code for a continuous function} (or just \textit{code}) iff
	\begin{enumerate}	
        \item if $\eta_1 \lhd \eta_2$, then $[T_{\eta_2}] \subseteq [T_{\eta_1}]$
		\item if $\eta_1 \perp \eta_2$, then $[T_{\eta_1}] \cap [T_{\eta_2}] = \emptyset$
		\item $\bigcup_{\eta \in 2^i} [T_\eta] = [T]$ for each $i < \kappa$.
	\end{enumerate} 
\end{definition}

\begin{theorem}
	If $\mathcal{P}$ is a ${<}\kappa$-closed forcing notion, then $\mathbf{\Sigma}^1_1(\kappa)$ properties (i.e. analytic properties in the sense of the projective hierarchy on $\kappa^\kappa$) are absolute between $V$ and $V^\mathcal{P}$.
\end{theorem}
\begin{proof}
	See \cite{khomskii_kappa_reals}.
\end{proof}

\begin{lemma}
	Let $\langle T, (T_\eta)_{\eta \in 2^{<\kappa}} \rangle$ be a code. Then there exists a unique continuous function $g_{\langle T, (T_\eta)_{\eta \in 2^{<\kappa}} \rangle} : [T] \to 2^\kappa$ such that 
	\[
	g_{\langle T, (T_\eta)_{\eta \in 2^{<\kappa}} \rangle}^{-1}([\eta]) = [T_\eta]
	\]
	for all $\eta \in 2^{<\kappa}$.
\end{lemma}

\begin{proof}
	If we set $g(y) := \bigcup\{\eta \in 2^{<\kappa}: y \in [T_\eta]\}$, then it is easy to see that $g : [T] \to 2^\kappa$ is a well-defined continuous function and $g^{-1}([\eta]) = [T_\eta]$ for all $\eta \in 2^{<\kappa}$.	Since $([\eta])_{\eta \in 2^{<\kappa}}$ forms a clopen basis of $2^\kappa$, uniqueness is given.
\end{proof}

On the other hand, if $g : Y \to 2^\kappa$ is a continuous function where $Y \subseteq 2^\kappa$ is closed, then $\langle T, (T_\eta)_{\eta \in 2^{<\kappa}} \rangle$ is a code for $g$, where $T$ and the $T_\eta$'s are trees with $[T] = Y$ and $[T_\eta] = g^{-1}([\eta])$.

\begin{definition}
	For codes $c, c'$ define $c \preccurlyeq c' :\Leftrightarrow g_c \subseteq g_{c'}$.
\end{definition}
Clearly $\preccurlyeq$ is reflexive and transitive.

\begin{definition} \label{def: uniformly continuous}
	A function $g : Y \to Z$ with $Y,Z \subseteq 2^\kappa$ is uniformly continuous iff
	\[
	\forall i < \kappa\, \exists j(i) < \kappa\, \forall x \in Y:\ g''([x {\restriction} j(i)] \cap Y) \subseteq [g(x) {\restriction} i] \cap Z.
	\]
	The map $i \mapsto j(i)$ is the \textit{modulus of continuity} of $g$.
\end{definition}

\begin{fact} \label{fact: continuity absoluteness}
	The following statements are $\mathbf{\Pi}_1^1(\kappa)$ and therefore absolute for ${<}\kappa$-closed forcing extensions:
	\begin{itemize}
		\item $c$ is a code for a continuous function
		\item ``$[T] = [T']$'' for trees $T,T'$ 
		\item ``$c \preccurlyeq c'$'' for codes $c,c'$
		\item ``$g_c$ is a total function'' for a code $c$
		\item ``$\ran (g_c) \subseteq [T]$'' for a code $c$ and a tree $T$ 
		\item ``$g_c$ is uniformly continuous with modulus of continuity $i \mapsto j(i)$'' for a code $c$
	\end{itemize}
\end{fact}

Let now $Y \subseteq 2^\kappa$ be closed and $g: Y \to 2^\kappa$ be continuous. The above thus yields a method to continuously and uniquely extend $g$ to $\tilde{g}: Y^{(V^\mathcal{P})} \to (2^\kappa)^{(V^{\mathcal{P}})}$ \footnote{Here $Y^{(V^\mathcal{P})}$ is defined as $[T]^{(V^\mathcal{P})}$ for a ground model tree $T$ with $[T] = Y$. By the above fact, this definition does not depend on the choice of $T$.} for a ${<}\kappa$-closed forcing $\mathcal{P}$. To do so, choose a code $c$ for $g$ and evaluate it in $V^{\mathcal{P}}$. It is easy to prove that $\tilde{g} = (g_c)^{V^\mathcal{P}}$ is an extension of $g$; also observe that $\tilde{g}$ is independent of the chosen code $c$, since the statement $c \preccurlyeq c'$ is absolute by the above fact. Furthermore, we note that $\tilde{g}$ is the unique continuous extension of $g$, since $Y$ is dense in $Y^{V^\mathcal{P}}$.

In the future we will not be making a notational distinction between $g$ and $\tilde{g}$.

\section{Second Proof} \label{sec: second proof}

In this section we will construct a model in which every $X \subseteq 2^\kappa$ of size $|2^\kappa|$ can be uniformly continuously mapped onto $2^\kappa$. The construction closely follows Corazza's approach \cite{corazza}.

We will consider the same forcing iteration $\langle \mathbb{P}_\alpha, \dot{\mathbb{Q}}_\beta : \alpha \leq \kappa^{++}, \beta < \kappa^{++} \rangle$ with ${\leq}\kappa$-support as in the previous section. Additionally, we also choose $\dot{\mathbb{Q}}_\alpha$ to be $\kappa$-Sacks forcing (i.e. $f_\alpha \equiv 2$) for $\alpha = 0$ and for $\alpha$ with cofinality $\kappa^+$. We still assume $V \models |2^\kappa| = \kappa^+$, but $\kappa$ is only required to be inaccessible this time.

Since the forcing iteration is identical to the one in the previous section, Theorem \ref{th: easy inclusion} holds and thus
\[
V^\mathbb{P} \models [2^\kappa]^{\leq \kappa^+} \subseteq \mathcal{SN}.
\]

The other direction of the proof hinges on a technical lemma.

\begin{lemma} \label{lem: second proof technical}
	Let $p \in \mathbb{P}, F \in [\kappa^{++}]^{<\kappa}, i < \kappa$, $Y \in [2^\kappa]^{<\kappa}$ and a $\mathbb{P}$-name $\dot{\tau}$ be given such that $p$ forces $\dot{\tau} \in 2^\kappa$ and $\dot{\tau} \notin V$. Then we may find an $X \in [2^\kappa]^{<\kappa}$ and a sequence $(q_j)_{j < \kappa}$ of conditions below $p$ such that 
	\begin{itemize}
		\item $\forall j_1 < j_2 < \kappa:\ q_{j_2} \leq_{F,i} q_{j_1} \leq_{F,i} p$,
		\item $\forall j < \kappa:\ q_j \Vdash \exists x \in \check{X}:\ \dot{\tau} {\restriction} j = x {\restriction} j$ and
		\item $X \cap Y = \emptyset$.
	\end{itemize}
\end{lemma}

\begin{proof}
	If necessary, we may strengthen $p$ twice in the following manner:
	\begin{itemize}
		\item Firstly, since $|Y| < \kappa$ and $p \Vdash \dot{\tau} \notin \check{Y}$, we may find a name $\dot{\ell}$ for an ordinal less than $\kappa$ such that
		\[
		p \Vdash \forall y \in \check{Y}:\ \dot{\tau} {\restriction} \dot{\ell} \neq y {\restriction} \dot{\ell}.
		\]
		Property B* enables us to find a $p' \leq_{F,i} p$ and $\ell^* < \kappa$ with
		\[
		\forall y \in Y: p' \Vdash \ \dot{\tau} {\restriction} \ell^* \neq y {\restriction} \ell^*
		\]
		by restricting a maximal antichain deciding $\dot{\ell}$.
		\item  Secondly, we can find a $p'' \leq_{F,i} p'$ that is $(F,i)$-bounded (see Definition \ref{def: bounded}).
	\end{itemize}
	
	So without loss of generality assume that $p$ already has both these properties. We construct the sequence $(q_j)_{j < \kappa}$ inductively:
	\begin{itemize}
		\item $j=0$: Set $q_0 := p$.
		\item $j \to j+1$: Since		
		\[
		D_{j+1} := \{r \leq q_j: r \text{ decides } \dot{\tau} {\restriction} (j+1)\}
		\] is open dense below $q_j$, we may apply Lemma \ref{lem: bounded -> complete} to $q_j, F,i$ and $D_{j+1}$ \footnote{$q_j$ is $(F,i)$-bounded by Fact \ref{fact: bounded open}.} to get $q_{j+1}$ and $(s_g^{j+1})_{g \in C_{j+1}}$ such that $q_{j+1}$ is $(D_{j+1},F,i)$-complete as witnessed by $(s_g^{j+1})_{g \in C_{j+1}}$. Note that we have $q_{j+1} \leq_{F,i} q_j \leq_{F,i} p$. 
		\item $\delta$ is a limit: Find a $\leq_{F,i}$-lower bound $\tilde{q}_\delta$ of $(q_\ell)_{\ell < \delta}$. Just as in the successor step, apply Lemma \ref{lem: bounded -> complete} to $\tilde{q_\delta}, F,i$ and 
		\[
		D_\delta := \{r \leq \tilde{q_\delta}: r \text{ decides } \dot{\tau} {\restriction} \delta\}
		\]
		to get $q_\delta$ and $(s_g^\delta)_{g \in C_\delta}$.
	\end{itemize}
	
	By Lemma \ref{lem: C decreasing} we know that $(C_j)_{j < \kappa}$ is a decreasing sequence of non-empty sets of size less than $\kappa$; as such, the sequence is eventually constant. Let $j^*$ be the index at which this happens.
	
	Now define
	\[
	X := \{x \in 2^\kappa: \exists g \in C_{j^*}\, \forall j < \kappa:\ s_{g}^j \Vdash \dot{\tau} {\restriction} j = x {\restriction} j\}
	\]
	
	For $g \in C_{j^*}$ the sequence $(s_g^j)_{j < \kappa}$ is decreasing by Lemma \ref{lem: C decreasing}. Hence each $g \in C_{j^*}$ successfully interprets $\dot{\tau}$ as some unique $x \in X$, i.e.
	\[
	\forall g \in C_{j^*}\, \exists !\, x \in X\, \forall j < \kappa:\ s_g^j \Vdash \dot{\tau} {\restriction} j = x {\restriction} j.
	\]
	Since $ \forall y \in Y:\ p \Vdash\dot{\tau} {\restriction} \ell^* \neq y {\restriction} \ell^*$, we know that $X \cap Y = \emptyset$.
	
	Suppose now that $j \geq j^*$ and $r \leq q_{j}$. Then $r$ is compatible with $s_{g}^j$ for some $g \in C_{j} = C_{j^*}$ and we can find a $t \leq r, s_{g}^j$. But then $\exists x \in X:\ t \Vdash \dot{\tau} {\restriction} j = x {\restriction} j$, so we can conclude 
	\[
	q_j \Vdash \exists x \in \check{X}:\ \dot{\tau} {\restriction} j = x {\restriction} j.
	\]

    Since $|Y| < \kappa$, we may for each $j < j^*$ pick an arbitrary $x_j \in [\sigma_j] \backslash Y$, where 
    \[
    q_j \Vdash \dot{\tau} {\restriction} j = \check{\sigma}_j
    \]
    and add those $x_j$ to $X$, thereby ensuring that
	\[
	q_j \Vdash \exists x \in \check{X}:\ \dot{\tau} {\restriction} j = x {\restriction} j
	\]
	holds for $j < j^*$ as well.
\end{proof}

We are now preparing to show that every new real $\dot{\tau}^G \in V^\mathbb{P}$ can be mapped onto the first Sacks real $\dot{s}_0$ via a uniformly continuous ground model function. In what follows we shall slightly abuse notation; for $p \in \mathbb{P}$ and a node $\eta \in p(0)$ denote by $p^{[\eta]}$ the condition that satisfies $p^{[\eta]}(0) = p(0)^{[\eta]}$ and $p^{[\eta]}(\beta) = p(\beta)$ for $\beta > 0$.

\begin{lemma} \label{lem: crucial lemma one step}
	Let $p \in \mathbb{P}, F \in [\kappa^{++}]^{<\kappa}$ and $i, \ell < \kappa$. Let furthermore a $\mathbb{P}$-name $\dot{\tau}$ be given such that $p$ forces $\dot{\tau} \in 2^\kappa$ and $\dot{\tau} \notin V$. Then we can find a $q \leq_{F,i} p$, an $\ell^* > \ell$ and a family $(A_\eta)_{\eta \in \splitt_i(p(0))}$ of non-empty, clopen sets with
	\begin{itemize}
 		\item if $\eta_1 \neq \eta_2$, then $A_{\eta_1} \cap A_{\eta_2} = \emptyset$,
		\item $A_\eta = \bigcup_{\nu \in S_\eta} [\nu]$ for some $S_\eta \subseteq 2^{\ell^*}$ and
		\item $q^{[\eta]} \Vdash \dot{\tau} \in A_\eta$.
	\end{itemize}
\end{lemma}

\begin{proof}
	Enumerate $\splitt_i(p(0))$ as $(\eta_k)_{k < \delta}$ with  $\delta < \kappa$. We inductively construct sequences $((t_j^k)_{j < \kappa})_{k < \delta}$ and a sequence of sets $(X_k)_{k < \delta}$: assuming that $X_m$ has been constructed for $m < k$, apply Lemma \ref{lem: second proof technical} to $p^{[\eta_k]}$ and $Y := \bigcup_{m < k} X_m$ to get a sequence of conditions $(t_j^k)_{j < \kappa}$ and a set $X_k$.
	
	Now let $\ell^* > \ell$ be an ordinal large enough such that whenever $j_1 \neq j_2$ for $j_1, j_2 < \delta$ and $x_1 \in X_{j_1}, x_2 \in X_{j_2}$ then $x_1 {\restriction} \ell^* \neq x_2 {\restriction} \ell^*$. This is possible, since the $(X_k)_{k < \delta}$ are disjoint and of size less than $\kappa$. This allows us to define 
	\[
	A_{\eta_k} := \bigcup_{x \in X_k} [x {\restriction} \ell^*].
	\]
	
	Now we glue the conditions $t^k_{\ell^*}$ together in the following way: Set 
	\[
	q(0) := \bigcup_{k < \delta} t^k_{\ell^*}(0)
	\]
	and for $\beta > 0$ define $q(\beta)$ inductively; assuming $q {\restriction} \beta$ has been defined, set 
	\[
	q(\beta) := t^{\dot{k}}_{\ell^*}(\beta), \text{ where $\dot{k}$ is the unique ordinal less than $\delta$ such that $t^{\dot{k}}_{\ell^*} {\restriction} \beta \in \dot{G}_\beta$}.
	\]
	Note that $(t_{\ell^*}^k {\restriction} \beta)_{k< \delta}$ is a maximal antichain below $q {\restriction} \beta$. For limit $\beta$ observe that $\supp(q {\restriction} \beta) \subseteq \bigcup_{k < \delta} \supp(t^k_{\ell^*} {\restriction} \beta)$. By Lemma \ref{lem: second proof technical} we know that $t_{\ell^*}^k \leq_{F,i} p^{[\eta_k]}$ for each $k < \delta$, and since $\splitt_i(p(0)) = \{\eta_k : k < \delta\}$, we can inductively conclude $q {\restriction} \beta \leq_{F \cap \beta,i} p {\restriction} \beta$ for all $\beta \leq \kappa^{++}$.
	
	To see the last claim, only note that $q^{[\eta]} = t_{\ell^*}^k$ for some $k < \delta$, therefore by Lemma \ref{lem: second proof technical} we have $t_{\ell^*}^k \Vdash \exists x \in \check{X_k}:\ \dot{\tau} {\restriction} \ell^* = x {\restriction} \ell^*$ and thus
	\[
	q^{[\eta]} \Vdash \dot{\tau} \in A_\eta
	\]
	by definition of $A_\eta$.
\end{proof}

\begin{remark} \label{rem: a_eta minimality}
	Without loss of generality, we may choose the $A_\eta$ in the previous lemma to be minimal in the following sense: for each $\nu \in 2^{\ell^*}$ we have $\nu \in S_\eta$ iff there exists a condition $r \leq q^{[\eta]}$ such that $r \Vdash \dot{\tau} \in [\nu]$.
\end{remark}

\begin{lemma} \label{lem: key}
	Let $p \in \mathbb{P}$ and a $\mathbb{P}$-name $\dot{\tau}$ be given such that $p$ forces $\dot{\tau} \in 2^\kappa$ and $\dot{\tau} \notin V$. Then there exists a $q \leq p$, a sequence $(\ell^*(i))_{i < \kappa}$ and a family $(A_\eta)_{\eta \in \splitt(q(0))}$ such that $A_\eta \subseteq 2^\kappa$ are non-empty, clopen and:
	\begin{itemize}
		\item if $\eta_1 \lhd \eta_2$, then $A_{\eta_2} \subseteq A_{\eta_1}$,
		\item if $\eta_1 \perp \eta_2$, then $A_{\eta_1} \cap A_{\eta_2} = \emptyset$,
  		\item if $\eta \in \splitt_i(q(0))$, then $A_\eta = \bigcup_{\nu \in S_\eta} [\nu]$ for some $S_\eta \subseteq 2^{\ell^*(i)}$ and
		\item $q^{[\eta]} \Vdash \dot{\tau} \in A_\eta$.
	\end{itemize}
\end{lemma}

\begin{proof}
	We shall construct a fusion sequence $\langle q_i, F_i: i < \kappa \rangle$ and a strictly increasing sequence $(\ell^*(i))_{i < \kappa}$ of ordinals less than $\kappa$ such that $q_{i+1}$ has the required properties for $(A_\eta)_{\eta \in \splitt_i(q_i(0))}$.
	\begin{itemize}
		\item $i=0$: Set $q_0 := p$ and $F_0 := \{0\}$.
		\item $i \to i+1$: Applying Lemma \ref{lem: crucial lemma one step} to $q_i, F_i, i$ and $\sup_{j < i}\ell^*(j)$ yields a $\tilde{q} \leq_{F_i,i} q_i$, an ordinal $\ell^*(i)$ and a family $(A^i_\eta)_{\eta \in \splitt_i(q_i(0))}$. Set $q_{i+1} := \tilde{q}$. Define $F_{i+1}$ with a bookkeeping argument.
		\item $\delta$ is a limit: Set $q_\delta$ to a fusion limit of $\langle q_j, F_j: j < \delta \rangle$ and $F_\delta := \bigcup_{j < \delta} F_j$.
	\end{itemize}
	
	Let now $q_\kappa$ be a fusion limit of the sequence $\langle q_i, F_i: i < \kappa \rangle$ and for $\eta \in \splitt(q_\kappa(0))$ define
	\[
	A_\eta := A_\eta^{i(\eta)},
	\]
	where $i(\eta)$ is the unique $i$ with $\eta \in \splitt_i(q_\kappa(0)) = \splitt_i(q_i(0))$. We claim $q_\kappa$ has the properties we are looking for: 
	\begin{itemize}
		\item The third property holds by Lemma \ref{lem: crucial lemma one step}.
		\item If we assume the $A_\eta^{i(\eta)}$ have been chosen minimal in each step as in Remark \ref{rem: a_eta minimality}, then the first property holds. To see this, take $\nu \lhd \eta$ and $\eta' \in S_\eta$, where $S_\eta$ is as stated in Lemma \ref{lem: crucial lemma one step}. By Remark \ref{rem: a_eta minimality} there is a condition $r \leq q_{i(\eta) + 1}^{[\eta]}$ such that $r \Vdash \dot{\tau} \in [\eta']$. But then $r \leq q_{i(\eta) + 1}^{[\eta]} \leq q_{i(\nu) + 1}^{[\nu]}$, and thus $\eta' {\restriction} \ell^*(i(\nu)) \in S_\nu$. Hence $A_\eta \subseteq A_\nu$.
		\item For the second property, let $\eta, \nu \in \splitt(q_\kappa(0))$ with $\eta \perp \nu$ be given. Without loss of generality assume $i(\nu) \leq i(\eta)$ and find an $\eta'$ with $\eta' \lhd \eta$ and $\eta' \in \splitt_{i(\nu)}(q_\kappa(0))$; by incompatibility we have $\eta' \neq \nu$. By the first property we have $A_{\eta} \subseteq A_{\eta'}$ and Lemma \ref{lem: crucial lemma one step} yields $A_{\eta'} \cap A_\nu = \emptyset$.
		\item To see the fourth property, let $\eta \in \splitt(q_\kappa(0))$. Then we have $q_\kappa^{[\eta]} \leq q_{i(\eta)+1}^{[\eta]}$ and therefore
		\[
		q_\kappa^{[\eta]} \Vdash \dot{\tau} \in A_\eta,
		\]
		as desired. \qedhere
	\end{itemize}
\end{proof}

The following lemma substitutes in for Tietze's Extension Theorem from the countable case in \cite{corazza}. Recall the notion of superclosure (page \pageref{def: superclosed}) and uniform continuity (Definition \ref{def: uniformly continuous}).
\begin{lemma} \label{lem: extension}
	Let $Y,Z \subseteq 2^\kappa$, where $Y$ is closed and $Z$ is superclosed, and let $g: Y \to Z$ be uniformly continuous. Then $g$ can be extended to a uniformly continuous function $\tilde{g} : 2^\kappa \to Z$ with the same modulus of continuity as $g$.
\end{lemma}

\begin{proof}
	The open set $2^\kappa \backslash Y$ can be be written as a union of basic open sets $\bigcup_{i < \delta} [\nu_i]$ with $\delta \leq \kappa, \nu_i \in 2^{\lambda_i}$ such that the $\nu_i$ are minimal, i.e. 
	\[
	\forall j < \lambda_i:\ [\nu_i {\restriction} j] \cap Y \neq \emptyset.
	\]
	In particular the sets $[\nu_i]$ are pairwise disjoint. We will define $\tilde{g}$ to extend $g$ and to be constant on each $[\nu_i]$.
	
	For $i < \delta$ define 
	\[
	S(i) := \{\eta \in 2^{<\kappa} : \exists j < \lambda_i : g''([\nu_i {\restriction} j ] \cap Y) \subseteq [\eta] \cap Z\}.
	\] 
	Clearly $S(i)$ consists of pairwise $\lhd$-compatible elements; furthermore, for each $\eta \in S(i)$ we have $[\eta] \cap Z \neq \emptyset$. Since $Z$ is superclosed \footnote{If $|S(i)| = \kappa$, then $[\bigcup S(i)]$ is not defined, so work with $\{\bigcup S(i)\}$ instead.}, we have $Z \cap [\bigcup S(i)] \neq \emptyset$. We may thus set $\tilde{g} {\restriction} [\nu_i]$ to be constant with an arbitrary, fixed value from $Z \cap [\bigcup S(i)]$.
	
	It remains to check that $\tilde{g} : 2^\kappa \to Z$ is uniformly continuous with the same modulus of continuity as $g$. To this end, let $i < \kappa$ and $x \in 2^\kappa$. Consider $y \in [x {\restriction} j(i)]$. 
	\begin{itemize}
		\item If $x \in Y$, the interesting case is $y \notin Y$, hence $y \in [\nu_\ell]$ for some $\ell < \delta$. But then $j(i) < \lambda_\ell$ and 
		\[
		g''([x {\restriction} j(i)] \cap Y) \subseteq [g(x) {\restriction} i] \cap Z,
		\]
		hence by definition $g(x) {\restriction} i \in S(\ell)$ and thus $\tilde{g}(y) \in [\bigcup S(\ell)] \cap Z \subseteq [\tilde{g}(x) {\restriction} i] \cap Z$.
		\item On the other hand, if $x \notin Y$, then $x$ is in $[\nu_k]$ for some $k < \delta$. 
		
		Now one possibility is $[x {\restriction} j(i)] \cap Y = \emptyset$, in which case $j(i) \geq \lambda_k$ and $\tilde{g}$ is constant on $[x {\restriction} j(i)]$, therefore $\tilde{g}(y) = \tilde{g}(x) \in [\tilde{g}(x) {\restriction} i] \cap Z$. 
		
		The other possibility is $[x {\restriction} j(i)] \cap Y \neq \emptyset$ and thus $j(i) < \lambda_k$, and there once again are two cases to be distinguished:
		\begin{itemize}
			\item If $y \in Y$, then $g(y) {\restriction} i \in S(k)$ and thus 
			\[
			\tilde{g}(y) = g(y) \in [g(y) {\restriction} i] \cap Z = [\tilde{g}(x) {\restriction} i] \cap Z.
			\] 
			\item On the other hand, if $y \notin Y$, then $y \in [\nu_\ell]$ for some $\ell < \delta$. Since $[y {\restriction} j(i)] \cap Y = [x {\restriction} j(i)] \cap Y \neq \emptyset$, we can also conclude $j(i) < \lambda_\ell$. This means that $S(k) \cap S(\ell)$ contains an $\eta$ with $\dom(\eta) = i$ (namely $g(x') {\restriction} i$ for some $x' \in [x {\restriction} j(i)] \cap Y$) and thus $\tilde{g}(y) \in [\eta] \cap Z = [\tilde{g}(x) {\restriction} i] \cap Z$. \qedhere
		\end{itemize}
	\end{itemize}
\end{proof}

A natural question the inquisitive reader might pose is the validity of Lemma \ref{lem: extension} in case of the additional “artificial” assumption of superclosure being dropped. Indeed, the statement no longer holds; in \cite{schlicht_luecke} the authors observe, for instance, that the closed subset $Y$ of $2^\kappa$ consisting of all sequences with finitely many zeroes is not a retract of $2^\kappa$ (and thus the identity $Y \to Y$ cannot be extended to a continuous function on $2^\kappa$).

\begin{theorem} \label{th: map tau onto s0}
	Let $p \in \mathbb{P}$ force $\dot{\tau} \in 2^\kappa$ and $\dot{\tau} \notin V$. Then there exists a $q \leq p$ and a uniformly continuous function $f^* : 2^\kappa \to [q(0)]$ in $V$ such that
	\[
	q \Vdash f^*(\dot{\tau}) = \dot{s}_0,
	\]
	where $\dot{s}_0$ denotes the first Sacks real.
\end{theorem}

\begin{proof}
	Lemma \ref{lem: key} yields a condition $q \leq p$, a sequence $(\ell^*(i))_{i < \kappa}$ and a family $(A_\eta)_{\eta \in \splitt(q(0))}$ of clopen sets. This family codes \footnote{If we define
		\[
		A'_\eta :=
		\begin{cases}
			A_\nu, \text{ where } \nu = \min\{\rho \in \splitt(q(0)): \eta \lhd \rho\} &\text{ for } \eta \in q(0) \\
			\emptyset & \text{ for } \eta \notin q(0)
		\end{cases}
		\]
		and set $c := \langle T, (T_\eta)_{\eta \in 2^{<\kappa}} \rangle$, where $[T] = Y$ and $[T_\eta] = A'_\eta \cap Y$, then it can be seen that $c$ is a code for a continuous function; to avoid abuse of notation, we could also be working with $c$ at this point instead.} a continuous function 
	\begin{gather*}
		f : Y \to [q(0)] \\
		y \mapsto \bigcup\{\eta: y \in A_\eta\}
	\end{gather*}
	defined on the closed set $Y = \bigcap_{i < \kappa} \bigcup_{\eta \in \splitt_i(q(0))} A_\eta$ \footnote{Note that the ${<}\kappa$-box topology is closed under intersections of size less than $\kappa$.}.
	
	We claim that $f$ is in fact uniformly continuous. To see this, let $i < \kappa$ and $x \in Y$. Choose $\eta$ such that $x \in A_\eta$ and $\eta \in \splitt_{i}(q(0))$. Recall that $A_\eta$ is of the form (see Lemma \ref{lem: key})
	\[
	A_\eta = \bigcup_{\nu \in S_\eta} [\nu].
	\]
	with $S_\eta \subseteq 2^{\ell^*(i)}$.
	Therefore we have
	\[
	f''([x {\restriction} \ell^*(i)]) \subseteq [\eta] \subseteq [f(x) {\restriction} i],
	\]
	since $i \leq \dom(\eta)$.
	
	Since the set $[q(0)]$ is superclosed, we can apply Lemma \ref{lem: extension} and extend $f$ to a uniformly continuous function $f^* : 2^\kappa \to [q(0)]$. Lastly, we have 
	\[
	q^{[\eta]} \Vdash \dot{\tau} \in A_\eta \subseteq (f^*)^{-1}([\eta])
	\]
	for each $\eta \in \splitt(q(0))$ and thus 
	\[
	q \Vdash f^*(\dot{\tau}) = \dot{s}_0.
	\]
\end{proof}

As in the classical case, every $\kappa$-Sacks condition can be decomposed into $|2^\kappa|$ many $\kappa$-Sacks conditions in a continuous way. The last auxiliary result we require formalizes this:

\begin{lemma} \label{lem: continuous reading}
	Let $p \in \mathbb{P}$ be a condition. Then there exists a uniformly continuous $g^* : [p(0)] \to 2^{\kappa}$ \footnote{Recall that $p(0) \subseteq 2^{<\kappa}$.} and for each $x \in 2^\kappa \cap V$ a condition $q_x \leq p$ such that 
	\[
	q_x \Vdash \check{x} = g^*(\dot{s}_0). 
	\]
\end{lemma}

\begin{proof}
	First we construct a function $e = (e_1, e_2) : p(0) \to 2^{<\kappa} \times 2^{<\kappa}$ with the following properties:
	\begin{itemize}
		\item $e$ is continuous and monotone increasing
		\item $e(\emptyset) = (\emptyset, \emptyset)$
		\item if $\eta \notin \splitt(p(0))$, then $e(\eta ^\frown i) = e(\eta)$
		\item if $\eta \in \splitt_j(p(0))$ and
		\begin{itemize}
			\item $j$ is a successor, then $e(\eta^\frown i) = (e_1(\eta)^\frown i, e_2(\eta))$
			\item $j=0$ or $j$ is a limit, then $e(\eta^\frown i) = (e_1(\eta), e_2(\eta)^\frown i)$.
		\end{itemize}
	\end{itemize}
	Define $\hat{g} = (\hat{g}_1, \hat{g}_2) : [p(0)] \to 2^{\kappa} \times 2^\kappa$ as $\hat{g}_k(b) = \bigcup\{e_k(b {\restriction} i) : i < \kappa\}$ for $k=1,2$. Since $[p(0)]$ is perfect, $\hat{g}$ is well-defined. Moreover, $\hat{g}$ maps the clopen basis $([\eta])_{\eta \in \splitt(p(0))}$ to a clopen basis of $2^\kappa \times 2^\kappa$, hence it is a homeomorphism.
	
	For $x \in 2^\kappa$ now set $q_x(0) := \{\eta \in 2^{<\kappa}:\ \exists y \in \hat{g}^{-1}(\{x\} \times 2^\kappa):\ \eta \lhd y\}$ and $q_x(\beta) = p(\beta)$ for $\beta > 0$. We claim that $q_x$ is a condition; it is sufficient to check that $q_x(0)$ is. We check \ref{ax: splitting unbounded}, \ref{ax: superclosed} and \ref{ax: splitting continuous}; the rest is left as an exercise for the reader.
	
	\begin{itemize}
		\item \ref{ax: splitting unbounded}: Since $\hat{g}$ is a homeomorphism, it follows that $\hat{g}^{-1}\left(\{x\} \times 2^\kappa\right)$ is a perfect set.
		\item \ref{ax: superclosed}: Let $\left(\eta_j\right)_{j<\delta}$ with $\eta_j \in q_x(0)$ be a strictly increasing sequence of length $\delta<\kappa$. Set $\eta:=\bigcup_{j<\delta} \eta_j$. It easily follows that $\nu \in q_x(0) \Leftrightarrow x \in [e_1(\nu)]$ for all $\nu \in 2^{<\kappa}$. As $e_1(\eta)=\bigcup \{e_1(\eta_j) : j < \delta\}$ we see that $x \in [e_1(\eta)]$, hence $\eta \in q_x(0)$.
		\item \ref{ax: splitting continuous}: Let $(\eta_j)_{j<\delta}$ be a strictly increasing sequence of length less than $\kappa$ such that $\eta_j \in \splitt(q_x(0))$. Again, set $\eta:=\bigcup_{j<\delta} \eta_j$. It follows that $\eta \in \splitt_\lambda(p(0))$ for some limit $\lambda$. But as $x \in [e_1(\eta)]$ and $e_1(\eta)=e_1(\eta^\frown i)$, we have $\eta ^\frown i \in q_x(0)$ for $i=0,1$, hence $\eta \in \splitt(q_x(0))$.
	\end{itemize}
	
	Clearly $q_x \leq p$. Now set $g^* := \hat{g}_1$. Then $g^*$ is uniformly continuous with modulus of continuity
	\[
	i \mapsto j(i) := \sup\{\dom(\nu) + 1: \nu \in \splitt_{i+1}(p(0))\}.
	\]
	Finally, we have $q_x \Vdash \check{x} = g^*(\dot{s}_0)$ by the definition of $q_x(0)$ and the absoluteness (see Fact \ref{fact: continuity absoluteness}) of the statement
	\[
	\ran(g^* \restriction [q_x(0)]) \subseteq \{x\}. \qedhere
	\]
\end{proof}

\begin{theorem}
	In $V^\mathbb{P}$, every subset $X$ of $2^\kappa$ of size $\kappa^{++}$ can be uniformly continuously mapped onto $2^\kappa$.
\end{theorem}

\begin{proof}
	Assume that $\dot{X}$ is a $\mathbb{P}$-name for a subset of $2^\kappa$ such that
	\[
	\Vdash_{\mathbb{P}} \forall h \text{ uniformly continuous function } \exists y \in 2^\kappa:\ y \notin h'' \dot{X}.
	\]
	We will show $\exists \alpha^* < \kappa^{++}:\ \Vdash_\mathbb{P} \dot{X} \subseteq V^{\mathbb{P}_{\alpha^*}}$, thus $\Vdash_\mathbb{P} |\dot{X}| \leq \kappa^+$.
	
	By our assumption on $\dot{X}$ and $\mathbb{P}$ satisfying the $\kappa^{++}$-c.c. we get
	\begin{align*}
		\forall \alpha < \kappa^{++}\, \forall& \dot{h} \text{ $\mathbb{P}_\alpha$-name for a uniformly continuous function } \\
		\exists& \beta < \kappa^{++}, \beta \geq \alpha \, \exists y \text{ $\mathbb{P}_\beta$-name for a real}:\ \Vdash_{\mathbb{P}} \dot{y} \notin \dot{h}'' \dot{X}.
	\end{align*}
	To increase legibility, let the ellipsis $(\dots)$ denote the four quantifications in the above statement. By interpreting the name $\dot{X}$ partially in the intermediate model $V^{\mathbb{P}_\beta}$, i.e. by identifying $\dot{X}$ with a canonical $\mathbb{P}_\beta$-name for a $\mathbb{P}_{\beta, \kappa^{++}}$-name, we get
	\[
	(\dots):\ \Vdash_{\mathbb{P}_\beta} \Vdash_{\mathbb{P}_{\beta, \kappa^{++}}} \dot{y} \notin \dot{h}'' \dot{X}.
	\]
	Keep in mind that $\dot{y}, \dot{h}$ are both $\mathbb{P}_\beta$-names, since $\beta \geq \alpha$.

	Without loss of generality assume that the function $\alpha \mapsto \beta(\alpha)$ maps to the minimal $\beta$ for which the statement holds. Observe that, crucially, since every continuous function $h : 2^\kappa \to 2^\kappa$ can be coded by an element of $2^\kappa$ (see Section \ref{sec: coding}), no new functions of the kind appear at stages of cofinality $>\kappa$ (Lemma \ref{lem: no collapse}). Therefore we can easily find a fixed point of the function $\alpha \mapsto \beta(\alpha)$ with cofinality $\kappa^+$; call it $\alpha^*$. For $\alpha^*$ we thus know that
	\[
	V^{\mathbb{P}_{\alpha^*}} \models \forall h \text{ uniformly continuous function } \exists y \in 2^\kappa:\ \Vdash_{\mathbb{P}_{\alpha^*, \kappa^{++}}} \check{y} \notin h'' \dot{X}.
	\]
	For the remainder of this proof we will be working within $V^{\mathbb{P}_{\alpha^*}}$. We wish to show $\Vdash_{\mathbb{P}_{\alpha^*, \kappa^{++}}} \dot{X} \subseteq V^{\mathbb{\mathbb{P}_{\alpha^*}}}$. 
	
	Let thus $p \in \mathbb{P}_{\alpha^*, \kappa^{++}}$ and $\dot{\tau}$ be a $\mathbb{P}_{\alpha^*, \kappa^{++}}$-name such that $p$ forces $\dot{\tau} \in 2^\kappa$ and $\dot{\tau} \notin V^{\mathbb{P}_{\alpha^*}}$. Theorem \ref{th: map tau onto s0} applied within $V^{\mathbb{P}_{\alpha^*}}$ (recall that the tail iteration $\mathbb{P}_{\alpha*, \kappa^{++}}$ has the same structure as the full iteration) yields a $q \leq p$ and a uniformly continuous function $f^*: 2^\kappa \to [q(0)]$ such that $q \Vdash_{\mathbb{P}_{\alpha^*, \kappa^{++}}} f^*(\dot{\tau}) = \dot{s_0}$. Likewise, Lemma \ref{lem: continuous reading} applied to $q$ gives us a uniformly continuous function $g^*: [q(0)] \to 2^\kappa$ and conditions $(q_x)_{x \in 2^\kappa \cap V^{\mathbb{P}_{\alpha^*}}}$ with $q_x \Vdash_{\mathbb{P}_{\alpha^*, \kappa^{++}}} \check{x} = g^*(\dot{s_0})$.
	
	Now let $x \in 2^\kappa \cap V^{\mathbb{P}_{\alpha^*}}$ be arbitrary. By construction we have $q_x \Vdash \check{x} = (g^*\circ f^*)(\dot{\tau})$. For the uniformly continuous function $(g^* \circ f^*) : 2^\kappa \to 2^\kappa$ we can by our assumption on $\alpha^*$ find a $y \in 2^\kappa \cap V^{\mathbb{P}_{\alpha^*}}$ with $\Vdash_{\mathbb{P}_{\alpha^*, \kappa^{++}}} \check{y} \notin (g^* \circ f^*)'' \dot{X}$. The condition $q_y$ thus forces $\dot{\tau} \notin \dot{X}$. Since $p$ and $\dot{\tau}$ were arbitrary, we may conclude
	\[
	\Vdash_{\mathbb{P}_{\alpha^*, \kappa^{++}}} \dot{X} \subseteq V^{\mathbb{\mathbb{P}_{\alpha^*}}}.
	\]
	Thus we have shown $V \models\ \Vdash_{\mathbb{P}_{\alpha^*}} \Vdash_{\mathbb{P}_{\alpha^*, \kappa^{++}}} \dot{X} \subseteq V^{\mathbb{\mathbb{P}_{\alpha^*}}}$, which finishes the proof.
\end{proof}

It is easy to see that the uniformly continuous image of a strong measure zero set remains strong measure zero; thus we have shown 
\[
V^\mathbb{P} \models \mathcal{SN} \subseteq [2^\kappa]^{\leq \kappa^+}.
\]

\begin{corollary}
	$V^{\mathbb{P}} \models \mathcal{SN} = [2^\kappa]^{\leq \kappa^+}$.
\end{corollary}

\section{Stationary Strong Measure Zero}

Finally, let us take a look at the following definition, introduced by Halko \cite{halko}:

\begin{definition}
	A set $X \subseteq 2^\kappa$ is called \textit{stationary strong measure zero} iff 
	\[
	\forall f \in \kappa^\kappa\, \exists (\eta_i)_{i < \kappa}:\ (\forall i < \kappa:\ \eta_i \in 2^{f(i)}) \wedge X \subseteq \bigcap_{cl \subseteq \kappa \text{ club}}\ \bigcup_{i \in cl} [\eta_i].
	\]
\end{definition}

So a set $X$ is stationary strong measure zero iff we can find coverings that cover every point of $X$ stationarily often. To motivate why this definition might be of interest, observe that even for regular strong measure zero sets, we can always find coverings that cover each point at least unboundedly often:

\begin{lemma} \label{lem: cover unboundedly}
	Let $X \subseteq 2^\kappa$ be strong measure zero. Then
	\[
	\forall f \in \kappa^\kappa:\ \exists (\eta_i)_{i < \kappa}:\ (\forall i < \kappa:\ \eta_i \in 2^{f(i)}) \wedge X \subseteq \bigcap_{j < \kappa} \bigcup_{i \geq j} [\eta_i].
	\]
\end{lemma}

\begin{proof}
	Partition $\kappa$ into sets $(U_i)_{i < \kappa}$, where each $U_i$ has size $\kappa$. For a challenge $f \in \kappa^\kappa$ and every $i < \kappa$ we can find coverings $(\eta_j^i)_{j \in U_i}$ that satisfy the challenge $(f(j))_{j \in U_i}$. But now $(\eta_j^i)_{j \in U_i, i < \kappa}$ has the property we are looking for.
\end{proof}

\begin{lemma} \label{lem: bounding ground club}
	Let $\mathcal{P}$ be a $\kappa^\kappa$-bounding forcing notion and $cl \in V^\mathcal{P}$ a club subset of $\kappa$. Then there is a club $cl' \in V$ with $cl' \subseteq cl$.
\end{lemma}

\begin{proof}
	In $V^\mathcal{P}$, let $h \in \kappa^\kappa$ enumerate $cl$ and $g \in \kappa^\kappa \cap V$ dominate $h$; note that $h$ is a continuous function. Define the functions 
	\begin{gather*}
		g'(0) = g(0),\ g'(\alpha + 1) = g(g'(\alpha)) \text{ and } g'(\lambda) = \sup_{i < \lambda} g'(i) \text{ for limit } \lambda \\
		h'(0) = h(0),\ h'(\alpha + 1) = h(g'(\alpha)) \text{ and } h'(\lambda) = \sup_{i < \lambda} h'(i) \text{ for limit } \lambda .
	\end{gather*}
	Let $\lambda$ be a limit. Then since $h'(\alpha) \leq g'(\alpha)$ and $g'(\alpha) \leq h'(\alpha + 1)$, we have $h'(\lambda) = g'(\lambda)$; furthermore we know $h'(\lambda) \in cl$ and $g \in V$, hence $(g'(\lambda))_{\lambda < \kappa, \lambda \text{ limit}}$ is a ground model club contained in $cl$.
\end{proof}

In the Corazza-type model from Section \ref{sec: second proof}, the notions of strong measure zero and stationary strong measure zero coincide.

\begin{theorem}
	$V^\mathbb{P} \models \forall X \subseteq 2^\kappa:\ X \in \mathcal{SN} \Leftrightarrow X \text{ is stationary strong measure zero}$.
\end{theorem}

\begin{proof}
	Modify the argument in Theorem \ref{th: easy inclusion aux} to show 
	\[
	V^\mathbb{P} \models \forall \alpha < \kappa^{++}: 2^\kappa \cap V^{\mathbb{P}_\alpha} \text{ is stationary strong measure zero}
	\]
	by instead showing the set
	\[
	D_{x,cl} := \{p \in \mathbb{Q}_\beta: \exists i \in cl:\ p \Vdash \dot{\sigma}(i) = x {\restriction} h(i)\}
	\]
	to be dense for every $x \in V^{\mathbb{P}_\alpha}$ and every ground model club $cl \subseteq \kappa$, where $\dot{\sigma}$ is as defined in Theorem \ref{th: easy inclusion aux}. As every club $cl \in V^\mathbb{P}$ contains a ground model club $cl'$ by Lemma \ref{lem: bounding ground club}, this is sufficient. To see that $D_{x,cl}$ is dense, merely note that for any $p \in \mathbb{Q}_\beta$ and $b \in [p] \cap V^{\mathbb{P}_\beta}$, the set 
	\[
	\{j < \kappa: b {\restriction} j \in \splitt(p)\}
	\]
	is a club and thus intersects $cl$.
\end{proof}

On the other hand, it follows from $|2^\kappa| = \kappa^+$ that there is a strong measure zero set which is not stationary strong measure zero.

\begin{theorem}
	Under $|2^\kappa|=\kappa^+$ there exists an $X \in \mathcal{SN}$ that is not stationary strong measure zero.
\end{theorem}

\begin{proof}
	First off, let us enumerate all strictly increasing functions in $\kappa^\kappa$ as $(f_\alpha)_{\alpha < \kappa^+}$ and likewise enumerate the set
	\[
	\mathcal{S} := \{\sigma \in (2^{<\kappa})^\kappa: \forall i < \kappa:\ \dom(\sigma(i)) = i + 1\}
	\]
	as $(\sigma_\alpha)_{\alpha < \kappa^+}$.
	
	We shall inductively construct three sequences $(x_\alpha)_{\alpha < \kappa^+}$, $(\tau_\alpha)_{\alpha < \kappa^+}$ and $(cl_\alpha)_{\alpha < \kappa^+}$ with the following properties:
	\begin{enumerate}[a)]
		\item $\forall \alpha < \kappa^+:\ x_\alpha \in 2^\kappa, \tau_\alpha \in (2^{<\kappa})^\kappa$ and $cl_\alpha$ is a club subset of $\kappa$ \label{enum: a}
		\item $\forall \alpha < \kappa^+\, \forall i < \kappa:\ \dom(\tau_\alpha(i)) = f_\alpha(i)$ \label{enum: b}
		\item $\forall \alpha < \kappa^+ \, \forall i < \kappa:\ \bigcup_{j \geq i} [\tau_\alpha(j)]$ is open dense \label{enum: c}
		\item $\forall \alpha < \kappa^+\, \forall \beta \leq \alpha:\ x_\beta \in \bigcup_{i < \kappa} [\tau_\alpha(i)]$ \label{enum: d}
		\item $\forall \beta < \kappa^+\, \forall \alpha < \beta:\ x_\beta \in \bigcup_{i < \kappa} [\tau_\alpha(i)]$ \label{enum: e}
		\item $\forall \alpha < \kappa^+:\ x_\alpha \notin \bigcup_{i \in cl_\alpha} [\sigma_\alpha(i)]$ \label{enum: f}
	\end{enumerate}
	
	Setting $X = \{x_\alpha: \alpha < \kappa^+\}$ yields a strong measure zero set (by \ref{enum: b}, \ref{enum: d} and \ref{enum: e}). However, $X$ is not stationary strong measure zero, since for the challenge $g: i \mapsto i+1$ property \ref{enum: f} ensures
	\[
	\forall \sigma \in \mathcal{S}\, \exists x \in X\, \exists cl \text{ club }:\ x \notin \bigcup_{i \in cl}[\sigma(i)].
	\]
	
	Suppose now, inductively, that $(x_\alpha)_{\alpha < \gamma}$, $(\tau_\alpha)_{\alpha < \gamma}$ and $(cl_\alpha)_{\alpha < \gamma}$ have been constructed for $\gamma < \kappa^+$. We wish to define $x_\gamma, \tau_\gamma$ and $cl_\gamma$. To this end, reindex $(x_\alpha)_{\alpha < \gamma}$ and $(\tau_\alpha)_{\alpha < \gamma}$ as $(\tilde{x}_{i+1})_{i < \kappa}$, $(\tilde{\tau}_{i+1})_{i < \kappa}$ \footnote{If $\gamma < \kappa$, use some $x$ and $\tau$ multiple times. For $\gamma = 0$ pick $x_0$ and $cl_0$ arbitrarily such that $x_0 \notin \bigcup_{i \in cl_0} [\sigma_0(i)]$.} and inductively construct $x_\gamma$ and $cl_\gamma$ by building up partial approximations $x_\gamma^j$ and $cl_\gamma^j$ for $j < \kappa$. Here $x_\gamma^j$ will  be a binary sequence of length at least $j+1$. 
	
	\begin{itemize}
		\item $j = 0$: Set $cl_\gamma^0 := 0$ and $x_\gamma^0 := \langle 1 - \sigma_\gamma(0)(0) \rangle$.
		\item $j \to j+1$: Since by assumption $\ell \mapsto \dom(\tilde{\tau}_{j+1}(\ell))$ is strictly increasing and $\bigcup_{\ell' \geq \ell} [\tilde{\tau}_{j+1}(\ell')]$ is open dense for all $\ell < \kappa$, we can find an $\ell^* > cl_\gamma^j$ with $x_\gamma^j \lhd \tilde{\tau}_{j+1}(\ell^*)$. Set $cl_\gamma^{j+1} := \dom(\tilde{\tau}_{j+1}(\ell^*))$ and $x_\gamma^{j+1} := \tilde{\tau}_{j+1}(\ell^*) ^\frown (1 - \sigma_\gamma(cl_\gamma^{j+1})(cl_\gamma^{j+1}))$.
		\item $\lambda$ is a limit: Set $cl_\gamma^\lambda := \sup_{j < \lambda} cl_\gamma^j$ and $x_\gamma^\lambda := (\bigcup_{j < \lambda} x_\gamma^j) ^\frown (1 - \sigma_\gamma(cl_\gamma^\lambda)(cl_\gamma^\lambda))$.
	\end{itemize}
	
	Now set $x_\gamma := \bigcup \{x_\gamma^j: j < \kappa\}$ and $cl_\gamma := \{cl_\gamma^j: j < \kappa\}$. In the construction we have ensured $x_\gamma \notin \bigcup_{j \in cl_\gamma} [\sigma_\gamma(j)]$ and $x_\gamma \in \bigcup_{j < \kappa} [\tilde{\tau}_{i+1}(j)]$ for all $i < \kappa$. Finally, it is elementary to construct $\tau_\gamma$ such that \ref{enum: b}, \ref{enum: c} and \ref{enum: d} holds.
\end{proof}

\newpage
\printbibliography
\end{document}